\documentclass{amsart}
\usepackage{stmaryrd}
\usepackage{mathtools}
\usepackage{color}
\usepackage{graphicx}
\usepackage{tikz-cd}
\usepackage[all,cmtip]{xy}
\usepackage{amsfonts}   
\usepackage{amsmath}
\usepackage{amsthm}
\usepackage{amssymb}
\usepackage{latexsym}
\newfont{\msam}{msam10}

\newtheorem{theorem}{Theorem}
\newtheorem{proposition}[theorem]{Proposition}
\newtheorem{corollary}[theorem]{Corollary}
\newtheorem{lemma}[theorem]{Lemma}
\theoremstyle{definition}

\newtheorem{remark}[theorem]{Remark}

\newtheorem{conjecture}[theorem]{Conjecture}

\newtheorem*{ack}{Acknowlegements}

\DeclareMathOperator{\spec}{\ensuremath{Spec}}

\DeclareMathOperator{\gr}{\ensuremath{gr}}

\let\nc\newcommand

\nc{\la}{\label}

\def\bthm{\begin{theorem}}
\def\ethm{\end{theorem}}
\def\blemma{\begin{lemma}}
\def\elemma{\end{lemma}}
\def\bproof{\begin{proof}}
\def\eproof{\end{proof}}
\def\bprop{\begin{proposition}}
\def\eprop{\end{proposition}}
\def\bcor{\begin{corollary}}
\def\ecor{\end{corollary}}

\def\Z{\mathbb{Z}}

\def\O{\mathcal{O}}

\def\N{\mathbb{N}}

\def\z{\mathbb{Z}}

\def\m{\mathfrak{m}}
\def\c{\mathbb{C}}
\def\CC{\mathcal{C}}

\def\gr#1{\mbox{\sf gr}(#1)}

\def\ni{\noindent}
\nc{\Hom}{{\rm{Hom}}}
\nc{\chara}{{\rm{char}}}
\nc{\Ext}{{\rm{Ext}}}
\nc{\HOM}{\underline{\rm{Hom}}}
\nc{\EXT}{\underline{\rm{Ext}}}
\nc{\TOR}{\underline{\rm{Tor}}}
\nc{\End}{{\rm{End}}}
\nc{\GL}{{\rm{GL}}}
\nc{\SL}{{\rm{SL}}}
\nc{\Rep}{{\rm{Rep}}}
\nc{\dlim}{\varinjlim}

\newcommand{\h}{{\mathfrak{h}}}

\newcommand{\Aut}{{\rm{Aut}_{\c}}}

\newcommand{\Der}{{\rm{Der}}}

\newcommand{\Tr}{{\rm{Tr}}}

\begin{document}
\begin{abstract}
 Associated with each finite subgroup $\Gamma$ of $\SL_2(\c)$
 there is a family of noncommutative algebras $O_\tau(\Gamma)$
 quantizing $\c^2/\!\!/\Gamma$. 
 Let $G_\Gamma$ be the group of $\Gamma$-equivariant automorphisms of $O_\tau$.
 In \cite{E} one of the authors defined and studied a natural action of $G_\Gamma$
  on certain quiver varieties associated with $\Gamma$. 
 He established a $G_\Gamma$-equivariant bijective correspondence between
quiver varieties and the space of isomorphism classes of $O_\tau$-ideals.
The main theorem in this paper states that when $\Gamma$ is a cyclic group, 
the action of $G_\Gamma$ on each quiver variety is transitive.
This generalizes an earlier result due to Berest and Wilson who showed
the transitivity of the automorphism group of the first Weyl algebra on
the Calogero-Moser spaces.
Our result has two important implications. 
First, it confirms the Bockland-Le Bruyn conjecture for cyclic quiver varieties.
Second, it will be used to give a complete classification of algebras 
  Morita equivalent to $O_\tau(\Gamma)$. 
 \end{abstract}

\title{ On transitive action on quiver varieties}
\author[X. Chen]{Xiaojun Chen}
\author[A. Eshmatov]{Alimjon Eshmatov}
\author[F. Eshmatov]{Farkhod Eshmatov}
\author[A. Tikaradze ]{Akaki Tikaradze}

\address{X. C. and F. E.: School of Mathematics, Sichuan University, Chengdu, Sichuan Province
610064, People's Republic of China}
\email{xjchen@scu.edu.cn, olimjon55@hotmail.com}

\address{A. E. and A. T.: Department of Mathematics and Statistics, University of Toledo, Toledo, OH 43606-3390, USA}
\email{alimjon.eshmatov@utoledo.edu, Akaki.Tikaradze@utoledo.edu}

\date{November 2, 2018}

\maketitle
\section{Introduction}
We begin by recalling a  result due 
to Y. Berest and G. Wilson
which motivated this paper.
The \textit{Calogero-Moser} space $\mathcal{C}_n$  ($n\ge 1$) is defined as
$$  \mathcal{C}_n:= \{ (X,Y) \in \mathrm{Mat}_{n}(\c) \times \mathrm{Mat}_{n}(\c) 
\, : \, \mathrm{rank}([X,Y] +I_n)=1\} /\!/ \mathrm{PGL}_n(\c)\, .
$$
Named after a class of integrable systems in classical mechanics (see \cite{KKS})
these algebraic varieties play an important role in several areas such as
geometry, integrable systems and representation theory
 (see, e.g., \cite{N}, \cite{EG}, \cite{Eti}, \cite{Go} and references therein). 
 They were studied in detail in \cite{W}, where it was shown (among other things) 
 that $\mathcal{C}_n$ is smooth affine irreducible complex 
 symplectic variety of dimension $2n$. Let $G$ be 
  the automorphism group of the free algebra $\c\langle x,y\rangle$
preserving the commutator $xy-yx$.
Then $G$ acts naturally on $\mathcal{C}_n$, by viewing  $\mathcal{C}_n$
as a subvariety of $n$-dimensional representations of  $\c\langle x,y\rangle$
(see Section \ref{sec3} for precise definition).
The main result of \cite{BW} 
states that each $\mathcal{C}_n$ carries a \textit{transitive}
$G$-action (see \textit{loc. cit.} Theorem~$1.3$).
%
It is worth pointing out that here the group $G$ is considered as a discrete group. 
However, since $\mathcal{C}_n$ 
is an affine variety and a $G$-orbit, 
these suggest that $G$ should possess a structure of an algebraic group.
Indeed, following Shafarevich \cite{Sh}, $G$ viewed as an automorphism group of $\c\langle x,y\rangle$ 
has a natural structure of an infinite-dimensional algebraic group such that the
$G$-action  on $\mathcal{C}_n$ is algebraic (see \cite{BEE} and \cite[Section $11$]{BW}).

Next, it is easy to see that $G$ preserves
  the symplectic structure on $\mathcal{C}_n$.
  This led Berest and Wilson to  conjecture that   $\mathcal{C}_n$ is
a coadjoint orbit of $G$ or possibly a central extension of $G$. 
 V. Ginzburg \cite{Gi} and R. Bockland-L. Le Bruyn \cite{BL}
suggested and proved a quiver generalization of this conjecture.
Let $Q=(I,E)$ be a quiver  with the vertex set $I$ and the edge set $E$.
 Then in \cite{N1}, for $\tau \in \c^{I}$  and $\beta \in \z^I_{\ge 0}$,
 Nakajima defined the \textit{quiver variety} 
$\mathfrak{M}^{\tau}(Q,\beta)$ associated to $Q$. 
Let $\bar{Q}$ be the double quiver of $Q$.
Then the commutator quotient  space $\mathbb{N}_{Q}:=\c\bar{Q}/[\c\bar{Q},\c\bar{Q}]$
of the path algebra $\c\bar{Q}$ has a Lie algebra structure, called the \textit{necklace Lie algebra}  
(see  \cite{Gi} and \cite{BL}). 
The elements of $\mathbb{N}_{Q}$, called the \textit{necklace words}, are spanned by  oriented cycles in 
$\bar{Q}$ considered up to cyclic permutations.
It is easy to see that $\mathbb{N}_{Q}$ is a central extension of $\Der_{\omega}(\c\bar{Q})$,
the Lie algebra of
derivations of $\c\bar{Q}$  vanishing at $\omega:=\sum_{a\in E}[a,a^*]$.
Explicitly, there is a short exact sequence of Lie algebras
$$ 0\longrightarrow S \longrightarrow \mathbb{N}_{Q} 
\longrightarrow \Der_{\omega}(\c\bar{Q}) \longrightarrow 0 \, , $$
where $S$ is the abelian Lie algebra on the vector space
spanned by vertices of $Q$. 
The main result  of \cite{Gi} and \cite{BL} states that  $\mathfrak{M}^{\tau}(Q,\beta)$ 
is a coadjoint orbit of $\mathbb{N}_{Q}$. 
This statement may be interpreted as follows. 
Let $G_\omega:=G_\omega(Q)$ be the group of $S$-algebra
automorphisms of $\c\bar{Q}$ preserving $\omega$.
 Then $\Der_{\omega}(\c\bar{Q})$ may be thought as the Lie algebra
of $G_{\omega}$ and 
the coadjointness result could be viewed 
as an infinitesimal transitivite action of the central extension  of $G_{\omega}$.

Next define an algebra $A_{Q}:=\c[\N_Q]\otimes \c\bar{Q}$ equipped with
 a natural \textit{trace map} $\Tr:A_{Q} \to \c[\N_Q]$
sending a closed path in $\c\bar{Q}$ to the corresponding necklace word
and an open path to zero. Let $\mathrm{Aut}_Q$ be trace
preserving $S$-algebra automorphisms of $A_{Q}$
fixing $\omega$.  Then $G$-transitivity on $\mathcal{C}_n$
 suggests the following conjecture for a natural action of $\mathrm{Aut}_Q$ 
on  $\mathfrak{M}^{\tau}(Q,\beta)$:
%
%
\begin{conjecture} \cite[Conjecture 5.8]{BL} 
\label{LBB}
For any quiver $Q$ the group
$\mathrm{Aut}_Q$ acts transitively on $\mathfrak{M}^{\tau}(Q,\beta)$.
\end{conjecture}
The aim of this paper is to prove this conjecture in the following case.
For $m \ge 1$, let $\tilde A_{m-1}$ be an affine Dynkin quiver of type $A$.
According to the McKay correspondence, the quiver $\tilde A_{m-1}$
corresponds to a cyclic group of order $m$.
Let $Q_{m}$ be the quiver obtained from $\tilde A_{m-1}$ by adding 
a vertex $\infty$ and an arrow $a_{\infty}:\infty \to 0$.
We will consider quiver varieties $\mathfrak{M}^{\tau}(Q_{m}, \beta)$.
These are natural generalizations of the Calogero-Moser spaces $\mathcal{C}_n$.
Like $\mathcal{C}_n$, the space $\mathfrak{M}^{\tau}(Q_{m}, \beta)$
is a smooth symplectic variety. 
Our main result states:
%
\begin{theorem}
\label{maintheorem}
The group $G_{\omega}$ acts transitively on $\mathfrak{M}^{\tau}(Q_{m}, \beta)$.
\end{theorem}
Theorem~\ref{maintheorem} 
implies Conjecture~\ref{LBB} for $Q=Q_{m}$, since $G_{\omega}(Q)$ 
is a subgroup of $\mathrm{Aut}_Q$ for any quiver $Q$.
Furthermore,  since $G_{\omega}=G$  for $m=1$, we obtain Berest-Wilson result 
as a special case of Theorem~\ref{maintheorem}.
However, our  proof  is different from that of Berest-Wilson, where they
identified $ \mathcal{C}_n$ with another space, the adelic Grassmannian
$\mathrm{Gr}^{ad}$ that paramertizes rational solutions of some integrable system 
(the KP hierarchy). They showed that this identification is $G$-equivariant, 
and used a non-trivial fact about this system to prove the transitivity. 

We now give an outline of our approach. Recall that to each quiver $Q=(E,I)$
one can associate the Weyl group $W$, which is a subgroup of $\mathrm{Aut}(\z^I)$
generated by reflections at loop-free vertices.
The group $W$ acts on the set of dimension vectors $\z^I_{\ge 0}$ and 
dually on the space of parameters $\c^I$. Then to each $s\in W$,
there is a bijective map $\mathcal{R}_{s}:\mathfrak{M}^{\tau}(Q,\gamma)
\to \mathfrak{M}^{s\cdot \tau}(Q,s(\gamma))$, called the {\it reflection functor}
(see \cite{N}, \cite{L}, \cite{CBH}). 
Later it has been shown that  $\mathcal{R}_{s}$
is an isomorphism of algebraic varieties (see \cite{M}).
We shall prove
%
\begin{theorem}
\label{thmGequiv}
The map $\mathcal{R}_{s}:\mathfrak{M}^{\tau}(Q_{m},\beta)
\to \mathfrak{M}^{s\cdot \tau}(Q_{m},s(\beta))$ is $G_{\omega}$-equivariant.
\end{theorem}
Recall that the \textit{generalized $n$-Calogero-Moser space}  $\mathcal{C}_{n,m}^{\tau}$ 
associated with $Q_{m}$ is the quiver variety $\mathfrak{M}^{\tau}(Q_{m}, \gamma)$ 
with the dimension vector $\gamma:=(1,n,...,n)$.
We show that for any $\beta$ there exists $s\in W$
so that $\mathcal{R}_{s}:\mathfrak{M}^{\tau}(Q_m,\beta)
\to \mathcal{C}_{n,m}^{\tau'}$ 
is an equivariant bijection for some $n$ and  $\tau'$.
Thus, using Theorem~\ref{thmGequiv}, we may reduce 
the transitivity on $\mathfrak{M}^{\tau}(Q_{m,\infty}, \beta)$
to that on the corresponding generalized Calogero-Moser space. It is worth pointing out 
that the idea of using the
reflection functors, in order to reduce the problem from general quiver variety
to a special one, is not new. For instance, in series of works 
\cite{Go, P}, the authors used $\c^*$-equivariance of
the reflection functors to obtain combinatorial description of $\c^*$-fixed points of quiver varieties. However, one 
novelty of our result is that, we prove equivariance of these bijection under a  ``larger" 
automorphism group. It seems to us 
this result might be of independent interest. 

 Let now $\Gamma$ be a cyclic group of order $m$, and let
 $\mathbf{\Gamma}_n=S_n \ltimes \Gamma^n$ be the 
 wreath product of $S_n$ and $\Gamma$. Then we consider  the associated symplectic
 reflection algebra $H_{0,c}(\mathbf{\Gamma}_n)$
 introduced by Etingof and Ginzburg \cite{EG}.  
They proved that $\O (\mathcal{C}_{n,m}^{\tau})$, the algebra
 of regular functions on $\mathcal{C}_{n,m}^{\tau}$, can be naturally identified with
$Z_c$,  the center of the algebra $H_{0,c}$.
 They showed this identification is an isomorphism of Poisson algebras.
Now the group $G_\omega$ has a subgroup generated by a family of
 Hamiltonian flows $\{ \, \Psi_{k, \mu}, \Phi_{k, \mu} \}_{k \in \z_{\geq 0}, \, \mu \in \c}$.
 By using a theorem of Bezrukavnikov and Etingof on certain completions of rational Cherednik algebras \cite{BE}, 
 we show that the corresponding 
 Hamiltonians are given by the above Poisson generators of $\O (\mathcal{C}_{n,m}^{\tau})$.
Since $\mathcal{C}_{n,m}^{\tau}$ is a smooth irreducible symplectic variety then
 the following is an easy consequence  of a well-known result on Poisson algebras. 

\begin{theorem}
\label{CMtran}
$G_\omega$ acts transitively on $\mathcal{C}_{n,m}^{\tau}$.
\end{theorem}

The paper is organized as follows. In Section~\ref{sec2},
besides introducing standard notations, we review some basic facts about 
quivers and quiver varieties.
In Section~\ref{sec3}, we introduce a group $G_m$ and 
define its action on $\mathfrak{M}^{\tau}(Q_{m}, \beta)$. 
It suffices for the proof of Theorem~\ref{maintheorem} to prove 
Theorem~\ref{thmGequiv} and Theorem~\ref{CMtran} for 
the subgroup $G_m$.
In Section~\ref{sec4} we give a brief exposition of reflection functors and present 
the proof of Theorem~\ref{thmGequiv}. 
In Section~\ref{sec5}, we establish transitivity result for generalized Calogero-Moser varieties (Theorem~\ref{CMtran}). 
For the convenience of the reader we recall some facts about Cherednik algebras and its
relationship with quiver varieties.
The paper ends with an appendix where we give
a description for the generating set of 
the ring $\mathcal{O}(\mathfrak{M}^{\tau}(Q_{m}, \beta))$, which could be 
 of independent interest. 

We would like to finish the introduction by mentioning one important application 
of Theorem~ \ref{maintheorem}.  Let  $\Gamma \subset \SL_2(\c)$
be a finite subgroup. Then it acts naturally on $\c \langle x, y \rangle $,
and we can form the crossed product  $R:=\c \langle x, y \rangle \ast \Gamma$.
For each $\tau \in Z(\c\Gamma)$, the center of the group  algebra $\c\Gamma$, let
$\mathcal{S}_\tau :=R / (xy-yx -\tau)$ and  
$O_\tau = e  \mathcal{S}_\tau e$ 
where $e$ is the symmetrizing idempotent $\sum_{g\in\Gamma} g/|\Gamma|$ in 
$\c\Gamma \subset \mathcal{S}_{\tau}$.  These algebras were introduced 
by Crawley-Boevey and Holland in \cite{CBH}.
The algebras $O_{\tau}$ can be viewed as a {\it quantization} of the coordinate ring of
the classical Kleinian singularity $\mathbb C^2/\!\!/\Gamma$.
Let $\mathcal{R}^\tau_{\Gamma}$ be the set of isomorphism classes of right ideals of $O_\tau$.
Then $\mathcal{R}^\tau_{\Gamma}$ can be naturally identified with the set of isomorphism classes of
finitely generated rank one projective $O_\tau$-modules. 
In \cite{E}, for $\Gamma \cong \z_m$, one of the authors constructed a bijective
correspondence between $\mathfrak{M}^\tau(Q_m)$ and $\mathcal{R}^\tau_{\Gamma}$,
and showed it is equivariant with respect to the action of $G_m\cong \Aut(O_\tau)$.
This generalizes an earlier result of Berest-Wilson on classification of $A_1$-ideals
\cite{BW2}. Using the bijective correspondence together with $G$-transitivity 
on $\mathcal{C}_n$, they gave complete classification of algebras Morita equivalent
to $A_1$.  We will address this question for $O_\tau$ in our subsequent paper
\cite{CEEF}.

\begin{ack}
X. Chen and F. Eshmatov are partially supported by NSFC No. 11671281.
A. Eshmatov is partially supported by University of Toledo Research Council Fellowship Award (No. 206184).
A. Tikaradze is very grateful to Gwyn Bellamy for helpful e-mail correspondence.
\end{ack}


%
%
\section{ Quiver varieties and Calogero-Moser spaces}
\label{sec2}
\subsection{Preliminaries on quivers}
Let $Q=(I,E)$ be a quiver, where $I$ is the set
of vertices and $E$ be the set of edges. For $a \in E$
such that $a : i \to j$, define $h(a):=j$ and $t(a):=i$.
The \textit{Ringel form} of $Q$ is the bilinear form
on $\mathbb{Z}^I$ is defined by
$$ \langle \alpha, \beta \rangle := \sum_{i \in I} \alpha_i \, \beta_i
-\sum_{a \in E} \alpha_{t(a)}\, \beta_{h(a)} \, . $$
The corresponding symmetric form is 
$(\alpha, \beta):= \langle \alpha, \beta \rangle + 
\langle \beta, \alpha  \rangle$.
For a loop-free quiver (no edge $a: i \to i$), the \textit{simple reflection} 
$s_i: \mathbb{Z}^I \to \mathbb{Z}^I$ is given by
\begin{equation}
\label{wdim}
s_i(\beta)\, = \, \beta-(\beta,\epsilon_i) \, \epsilon_i\, ,
\end{equation}
where $ \epsilon_i$ is the $i$th coordinate vector.
The \textit{dual reflection} $r_i: \mathbb{C}^I \to \mathbb{C}^I$
is 
\begin{equation}
\label{wpar}
(r_i \tau)_j= \tau_j- (\epsilon_i, \epsilon_j) \tau_i \, ,
\end{equation}
and it satisfies $r_i \tau \cdot \beta= \tau \cdot s_i(\beta)$.
Then the \textit{Weyl group} of $Q$ is the subgroup of automorphisms of $\z^I$ 
generated by $s_i$'s.  

The \textit{simple roots} of $Q$ are the coordinate vectors of loop-free vertices,
and a  \textit{real root} is the image of a simple root under the Weyl group.
A real root $\beta \in \z^{I}$ is positive (respectively negative) if  
$\beta \in \mathbb{N}^I$ (respectively $-\beta \in \mathbb{N}^I$). 
For an element $\beta= \sum_{i} k_i \epsilon_i \in \z^{I}$
the \textit{support}, denoted by $ {\rm supp}(\beta)$, is the subgraph of $Q$ consisting of vertices $i$
for which $k_i \neq 0$ and all the edges joining these vertices.
  The \textit{fundamental region} is 
$$ F:=\{ \beta \in \mathbb{N}^I-\underline{0}\, |\, (\beta, \epsilon_i) \le 0,\, 
\rm{supp}(\beta) \mbox{ is connected} \} \, .$$
An element $\alpha \in \z^I$ is an \textit{imaginary root} if 
it is of the form $w\beta$ or $-w\beta$ for some $w \in W$ and
$\beta \in F$. Any \textit{root} is either real or imaginary.
We say $\lambda \in \c^I$ is \textit{generic}
if $\lambda \cdot \alpha \neq 0$ for all roots $\alpha$ of $Q$.
%
%

%
\subsection{Quiver varieties}
%
Let $Q=(I,E)$ be a quiver. 
The double $\bar{Q}$ of $Q$ is the quiver obtained by adjoining an arrow 
$a^* \, : \, j \rightarrow i$ for each arrow $a\, : \, i \rightarrow \, j$ in $Q$.
We denote by $\c \bar{Q}$ its path algebra: the associative algebra with basis 
 paths in $\bar{Q}$, including trivial paths $e_i$.
Following \cite{CBH}, for $ \tau \in \c^I$, the 
\textit{deformed preprojective algebra}  is defined to be
$$
\Pi^{\tau} (Q) = \c \bar{Q} \, \bigg/ \, \bigg(\sum_{a \in E} [a,a^*] - \sum_{i \in I} \tau_i e_i\bigg) \, .
$$
The relation $\sum_{a \in E} [a,a^*] - \sum_{i \in I} \tau_i e_i$ can 
be replaced by an equivalent set of relations:
$$
\sum_{\substack{a \in E \\ h(a)=i}} aa^* - \sum_{\substack{a \in E 
\\ t(a)=i}} a^* a -  \tau_i e_i, \quad \quad i \in I \, .
$$
%
%
%
%
%
%
%
%
The space of all representations of $Q$ of dimension vector $\beta \in \mathbb{N}^I$ 
is given by 
$$
{\rm Rep} (Q, \beta) = \bigoplus_{a \in E} {\rm Mat } (\beta_{h(a)}
 \times \beta_{t(a)}, \c ) \, ,$$
%
where the isomorphism classes correspond to orbits of group
$$
G(\beta) = \Big( \prod_{i \in I} {\rm GL} (\beta_i , \c) \Big) \slash \c^{\times} 
$$
acting by conjugation. Using the trace pairing there is an identification of the
cotangent bundle
$$
{\rm T^*  Rep} (Q, \beta) \, \cong \, {\rm Rep} ( \bar{Q}, \beta) \, 
$$
Therefore the  space  ${\rm Rep} ( \bar{Q}, \beta)$ has a natural symplectic structure
$$
\omega = \sum_{a \in E} {\rm tr} ( d X_a \wedge d X_{a^*} ) ,
$$ 
invariant under $G(\beta)$.  We  define the moment map 
$\mu_{\beta} \, : \, {\rm Rep} (\bar{Q}, \beta) \, \rightarrow \, {\rm End} (\beta)_0$ by
$$
\mu_{\beta} (X)_i = \sum_{\substack{a \in E \\ h(a)=i}} X_a X_{a^*} -
 \sum_{\substack{a \in E \\ t(a)=i}} X_{a^*} X_a \, , 
$$
where
$$
{\rm End} (\beta)_0  = \Big\{ (\theta_i) \, \big | \, \sum_{i \in I}
 {\rm tr} (\theta_i) =0 \Big\} \subseteq {\rm End} (\beta)  = \oplus_{i \in I} {\rm Mat } (\beta_i , \c).
$$
If one uses the trace pairing to identify ${\rm End} (\beta)_0 $ with the dual of the 
Lie algebra of $G(\beta)$, then this is a moment map in the usual sense.
 Now if $\tau \cdot \beta :=\sum_{i\in I} \tau_i \beta_i=0 $ then 
$ \mu_{\beta}^{-1} (\tau)$ can be identified with ${\rm Rep} (\Pi^{\lambda}, \alpha)$,
the space of (left) $\Pi^{\tau}(Q)$-modules of dimension vector $\beta$.
Explicitly, it is the subvariety of ${\rm Rep} (\bar{Q}, \beta)$ defined by
\begin{equation*}
\sum_{\substack{a \in E \\ h(a)=i}} X_a X_{a^*} - 
\sum_{\substack{a \in E \\ t(a)=i}} X_{a^*} X_a \,  = \, \tau_i I_{\beta_i} \quad \quad i \in I.
\end{equation*}
For $\tau \in \c^{I}$ and $\beta \in  \mathbb{N}^I$, the \textit{quiver variety}
associated to $Q$ is the GIT quotient  
\begin{equation*}
\mathfrak{M}^{\tau}(Q, \beta):={\rm Rep} (\Pi^{\tau}, \beta)  \sslash G(\beta) = 
\mu^{-1}_{\beta} (\tau)\sslash G(\beta) \, .
\end{equation*}
%
Thus, the space $\mathfrak{M}^{\tau}(Q, \beta)$ classifies isomorphism classes of 
semisimple representations of the algebra $\Pi^{\tau}(Q)$ of dimension $\beta$, and the orbits 
on which  $G(\beta)$ acts freely correspond to simple representations.

Now we recall a description of possible dimension vectors of simple representations. 
Let $p(\beta):=1-\langle \beta, \beta \rangle$ be a function on $\z^I$. 
For $\tau \in \c^I$, let $\Sigma_\tau$ be the set of positive roots $\beta$
such that $\tau \cdot \beta=0$, and $p(\beta) > \sum_{t=1}^r p(\beta^{(t)})$
for any decomposition $\beta=\beta^{(1)}+...+\beta^{(r)}$ with $r \ge 2$,
$\beta^{(t)}$ a positive root with $\tau \cdot \beta^{(t)}=0$ for all $t$.
The following is  the main result of  \cite{CB}:
\begin{theorem}
\label{CBdim}
 $\Pi^{\tau}(Q)$ has a simple 
$\beta$-dimensional representation  if and only if $\beta \in \Sigma_\tau$. In this case, 
$\mathfrak{M}^{\tau}(Q,\beta)$ is a reduced and irreducible
scheme of dimension $2p(\beta)$.
\end{theorem}
For $\beta \in \Sigma_\tau$ there is a natural generating 
set of $\mathcal{O}(\mathfrak{M}^{\tau}(Q,\beta))$, the algebra of regular functions on 
$\mathfrak{M}^{\tau}(Q,\beta)$.
Let $(X_a)_{a\in \bar{Q}} \in \mathrm{Rep}(\bar{Q},\beta)$, and let 
$a_1\,...\,a_m $ be a path in $\c\bar{Q}$ that starts and ends 
at the same vertex. Then the function $\Tr(a_1\,...\,a_m )$
is invariant under the action $\GL(\beta)$.
By the classical theorem of Le Bruyn-Procesi,
such traces generate the algebra $\c[\mathrm{Rep}(\bar{Q},\beta)]^{\GL(\beta)}$. 
Moreover since $\GL(\beta)$ is reductive,
the algebra $\mathcal{O}(\mathfrak{M}^{\tau}(Q, \beta))$
is also generated by the same trace functions.

%
\subsection{Cyclic quiver case}
Let $\tilde A_{m-1}$ be the extended Dynkin quiver of type $A$.
It is  a quiver with $m$ vertices $\Z / m \Z = \{ 0,1,...,m-1 \}$ and $m$ arrows $a_i : i+1 \rightarrow i $.
Explicitly, $\tilde A_{m-1}$ is
$$
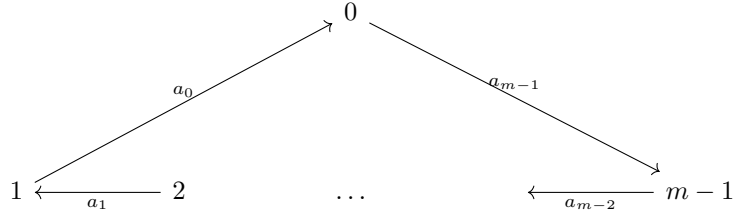
\begin{tikzcd}[ampersand replacement=\&] 
\& \& \& \& 0 \arrow{dddrrrr}[above]{a_{m-1}}\& \& \& \&  \\ 
\\
 \\
1  \arrow{uuurrrr}[above]{a_0}  \&  \&    \arrow{ll}[below]{a_1} 2 \& \& \ldots \& \& {}   
\& \& m-1 \arrow{ll}[below]{a_{m-2}} \\
\end{tikzcd}
$$
Let $W_{m-1}$ be the Weyl group of $\tilde{A}_{m-1}$
with presentation
$$
W_{m-1}= \big<s_0,...,s_{m-1}\, | \, s_i^2=1\, , \quad s_i \,s_{i+1}\,s_i=s_{i+1}\,s_i\,s_{i+1}
\, \mbox { for  }  \,  i \in \z/m\z \big> \, .
$$
The group $W_{m-1}$ acts on $\z^m$  by
$$ s_i(n_0, ...,n_{m-1})=(n_0,..., n_{i-1}+n_{i+1}-n_i,...,n_{m-1})
\, \mbox{  for  }\, i \in \z/m\z \ . $$
 The simple roots of $\tilde{A}_{m-1}$ are $\{ \epsilon_i \, | \,  i \in \z / m\z \}$, with
the fundamental region $F=\mathbb{N} \delta$, where 
$\delta=(1,...,1)$. 
%
%
%

Let $Q_m=(I_\infty, E_{\infty})$ be the quiver obtained 
from $\tilde{A}_{m-1}$ by adding an extra vertex, denoted by $\infty$, and an arrow
$a_{\infty}: \infty \to 0$. 
Namely, $Q_m$ is
$$
 
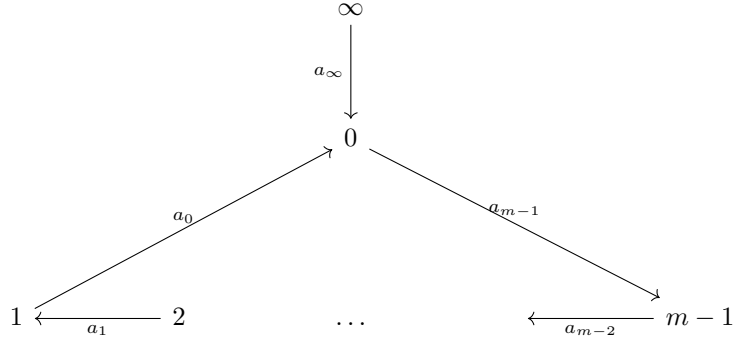
\begin{tikzcd}[ampersand replacement=\&] 
\& \& \& \& \infty \arrow{dd}[left]{a_{\infty}}  \& \& \& \&  \\
\\
\& \& \& \& 0  \arrow{dddrrrr}[above]{a_{m-1}}\& \& \& \&  \\ 
\\
 \\
1  \arrow {uuurrrr}[above]{a_0}  \&  \&    \arrow{ll}[below]{a_1} 2 \& \& \ldots \& \& {}  
 \& \& m-1 \arrow{ll}[below]{a_{m-2}} \\
\end{tikzcd}  \\
$$
Hence we have $I_{\infty}=\{0,1,...,m-1, \infty\} $ and $E_{\infty}=\{a_0,...,a_{m-1}, a_{\infty}$\} .

\begin{remark}
Note that the quiver obtained by extending  $\tilde{A}_{m-1}$ from any other than $0$ vertex 
would be equivalent to $Q_m$. In particular, various constructions for  $Q_m$ such
as deformed preprojective algebras and quiver varieties associated to $Q_m$ are 
independent of the choice of such extension (up to isomorphism).
\end{remark}

%
Let $W_{\infty}$ be the subgroup of the Weyl group
of $Q_m$ fixing vertex $\infty$. It is clear that $W_{\infty}$ is isomorphic
to $W_{m-1}$ and it acts on $\z^{m+1}$ as follows
\begin{eqnarray}
\label{wact1}
&& s_0 (n_{\infty},n_0,...,n_{m-1})=(n_{\infty},\,n_{\infty}+n_1+n_{m-1}-n_0,\, n_1,...,n_{m-1}) , \\
\label{wact2}
&& s_i (n_{\infty},n_0,...,n_{m-1})=(n_{\infty},\,n_0,..., n_{i-1}+n_{i+1}-n_i,...,\,n_{m-1})  ,
\end{eqnarray}
for $1\le i \le m-1$.
%
%

 For $\alpha \in \z^m$, $\lambda =(\lambda_i) \in \c^m$ let 
$\lambda_{\infty}:=-\lambda \cdot \alpha$. Then we consider
the deformed preprojective algebra $\Pi^{\tau}(Q_m)$ and 
 the quiver variety $\mathfrak{M}^{\tau}(Q_m,\beta)$, where 
$\tau:=(\lambda_{\infty},\lambda) \in \c^{m+1}$
and $\beta=(1,\alpha) \in \z^{m+1}$ (note $\beta \cdot \tau=0$). 
A point of $\mathfrak{M}^{\tau}(Q_m,\beta)$
is given by the following representation
of the double quiver $\bar{Q}_{\infty}$:

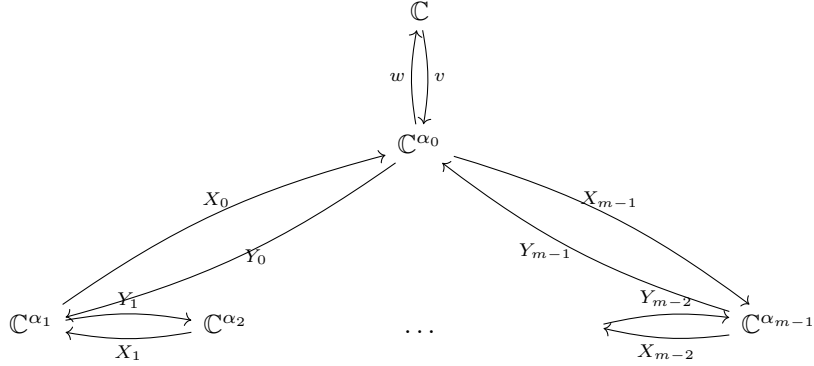
 \begin{tikzcd}[ampersand replacement=\&] 
\& \& \& \& \c  \arrow[bend left=10,swap]{dd}[right]{v} \& \& \& \&  \\
\\
\& \& \& \& \c^{\alpha_0} \arrow[bend left=10,swap]{uu}[left]{w} \arrow[bend left=10,swap]{dddllll}[right]{Y_0}  \arrow[bend left=10,swap]{dddrrrr}[above]{X_{m-1}}\& \& \& \&  \\ 
\\
 \\
\c^{\alpha_1}  \arrow[bend left=10,swap]{uuurrrr}[above]{X_0}   \arrow[bend left=10,swap]{rr}[above]{Y_1} \&  \&    \arrow[bend left=10,swap]{ll}[below]{X_1} \c^{\alpha_2} \& \& \ldots \& \& {}  \arrow[bend left=10,swap]{rr}[above]{Y_{m-2}} \& \& \c^{\alpha_{m-1}} \arrow[bend left=10,swap]{ll}[below]{X_{m-2}} \arrow[bend left=10,swap]{uuullll}[left]{Y_{m-1}}\\
\end{tikzcd}  \\
More explicitly, $\mathfrak{M}^{\tau}(Q_m,\beta)$ consists
$G(\alpha) = \prod_{i=0}^{m-1} \GL(\alpha_i)/\c^*$-orbits of tuples $(X_i, Y_i, v, w)$, 
where $X_{i}:=X_{a_i}$, $Y_i:=X_{a_i^*}$ for $i=0, 1,...,m-1$, $v:=X_{a_\infty}$,
$w:=X_{a_\infty^*}$,
satisfying 
\begin{eqnarray}
\label{qv1}
X_0 Y_0 - Y_{m-1} X_{m-1} +vw = \lambda_0 \,  I_{\alpha_0}  \\
\label{qv2}
X_i Y_i - Y_{i-1} X_{i-1} = \lambda_i \, I_{\alpha_i} \quad i \neq 0 \\
\label{qv3}
-wv =  \lambda_{\infty}    
 \end{eqnarray}
and the $G(\alpha)$-action is given by 
$$(g_0,...,g_{m-1}).(X_i,Y_i,v,w)\, :=\, (g_i\, X_i\, g_{i+1}^{-1}, g_{i+1}\,Y_i\,g_i^{-1}, g_0\,v, w\,g_0^{-1})\, .$$
Alternatively, if we set $X$ and $Y$ to be
\begin{equation}
 \label{newdta}
\begin{bmatrix} 0 & X_0 & 0 & \ldots & 0 \\ 
0 & 0 & X_1 & \ldots  & 0 \\
0 & 0 & 0 & \ddots & 0 \\
\vdots & \vdots & \ddots & \ddots & X_{m-2} \\
X_{m-1} & 0 & \ldots & 0 & 0
\end{bmatrix}  \  , \
\begin{bmatrix} 0 & 0 & 0 & \ldots & Y_{m-1} \\ 
Y_0 & 0 & 0 & \ldots & 0 \\
0 & Y_1 & 0 & \ddots  & \vdots \\
\vdots & \vdots & \ddots & \ddots & 0 \\
0 & 0 & \ldots & Y_{m-2} & 0
\end{bmatrix} 
\end{equation}
and
\begin{equation*}
 \bold{v}\,:=\,(v^T, 0_{\alpha_2},\ldots, 0_{\alpha_{m-1}})^T\, , \, 
 \bold{w}\,:=\, (w, 0_{\alpha_2},\ldots, 0_{\alpha_{m-1}})\, , 
 \end{equation*}
 \begin{equation*}
\Lambda\,:=\, \mathrm{Diag}[\,\lambda_0 I_{\alpha_0}, \ldots, \lambda_{m-1}I_{\alpha_{m-1}}\,],
\end{equation*}
then \eqref{qv1}-\eqref{qv3}  is equivalent to
\begin{equation}
\label{neweq}
XY - YX + \bold{v} \bold{w}\, =\, \Lambda \, , \, \bold{w} \bold{v}=\lambda_{\infty} \, .
\end{equation}
Thus, points of $\mathfrak{M}^{\tau}(Q_m,\beta)$ can be presented by classes
of quadruples $(X,Y, \bold{v},\bold{w})$ satisfying \eqref{neweq}.
The following
is a direct consequence of Theorem~\ref{CBdim}.
\begin{lemma}
\label{lemmamain}
Let $\lambda \in \c^m$ be generic for $\tilde{A}_{m-1}$, let $\alpha \in \z^{m}$,
and let $\tau=(-\lambda \cdot \alpha,\lambda)$. Then:
\begin{enumerate}
\item[(i)] Every $\Pi^{\tau}$-module of dimension vector $\beta=(1,\alpha)$ is simple.\\
\item[(ii)] The vector $\beta=(1,\alpha) \in \Sigma_{\tau}$  if and only if $\beta$ is a positive root of $Q_m$.
 In this case,
$$ \mathrm{dim}_{\c} \, \mathfrak{M}^{\tau}(Q_m, \beta)
=2p(\beta)=2 \alpha_0 - \sum_{i \in \z/m\z} (\alpha_i-\alpha_{i+1})^2 \, .$$ 
\end{enumerate}
\end{lemma}
\begin{proof}
The proof is similar  to that of \cite[Proposition~3]{BCE}.
\end{proof}

The quiver variety for $\alpha=(n,...,n)\in \mathbb{N}^m$, denoted
by $\mathcal{C}^\tau_n(Q_m)$,  is called 
\textit{the n-th Calogero-Moser space} associated with $Q_m$. Explicitly, we have
\begin{equation*}
 \mathcal{C}_{n}^{\tau}(Q_m):=\mathfrak{M}^{\tau}(Q_m, \beta)\, , \quad \beta=(1,n,...,n)\, ,
\quad \tau=\bigg(-n\sum_{i=0}^{m-1} \lambda_i, \lambda_0,...,\lambda_{m-1}\bigg)\,.
\end{equation*}
It was shown by Etingof and Ginzburg \cite[Proposition~11.11]{EG} that $ \mathcal{C}_{n}^{\tau}(Q_m)$
is a smooth symplectic affine algebraic variety of dimension $2n$. These varieties are
natural generalization of the classical Calogero-Moser spaces corresponding to the case $m=1$.
More explicitly, in this case $\mathcal{C}_{n}^{\tau}(Q_m) \cong \mathcal{C}_{n}$ for all $\tau \neq 0$ where
\begin{eqnarray*}
\mathcal{C}_n\, &:=& \, \{ (X, Y, v, w) \, | \, X, Y \in \End(\c^n), \, v \in \Hom(\c^n, \c),\\[0.15cm]
&& \, w \in \Hom(\c, \c^n) , \, XY -YX + \mathrm{Id}_n \, =\,  vw \}\,  /\!\!\!\ / \GL_n(\c). \nonumber
\end{eqnarray*}
These spaces have been studied extensively by Wilson \cite{W}.
%
%
%
%
%
%

The  generating set for general quiver varieties,
 discussed at the end of the previous section, can be significantly  
reduced for $\mathfrak{M}^{\tau}(Q_m, \beta)$.
Indeed for $\lambda \in \c^m$ generic, we have $\Pi^{\tau}e_0 \Pi^{\tau}=\Pi^{\tau}$
(see \cite[Lemma 8.5]{CBH}), hence any closed path in $\Pi^{\tau}$
is a sum of paths containing vertex $0$. Since the trace function
is invariant under a cyclic permutation, we may only consider
traces of paths starting and ending at $0$. 
Let $(X_i,Y_i,v, w)$ be a representative of a point  in
 $\mathfrak{M}^{\tau}(Q_m,\beta)$.  
 We set
 $$ A:=Y_{m-1} \ldots Y_1 Y_0\, , \quad B:= X_0 X_1 \ldots X_{m-1}\ , $$
$$C_0:=\mathrm{Id}_{n_0}\, , \quad C_k:=Y_{m-1} \ldots Y_{m-k}\, X_{m-k} \ldots X_{m-1} .$$
 Note that each of them is a closed path at $0$.  These are special paths starting with $Y_i$'s
 and then followed by $X_j$'s. We consider the following functions
\begin{equation}
\label{GHgen}
G^{i,j}_k:=\Tr(A^iC_kB^j)\, , \quad H^{i,j}_k:=w(A^iC_kB^j)v
\end{equation}
 on $\mathfrak{M}^{\tau}(Q_m,\beta)$.
 If we view $(X_i,Y_i,v,w)$  in terms of $(X,Y,\bold{v},\bold{w})$ 
 as in \eqref{newdta} and \eqref{neweq}, then they have the following presentation
 \begin{equation*}
G^{i,j}_k=\mathrm {Tr} (Y^{mi+k} X^{mj+k})\, , 
\quad H^{i,j}_k=\bold{w}(Y^{mi+k} X^{mj+k})\bold{v}\, .
\end{equation*}
Hence, by the Cayley-Hamilton theorem, it suffices to consider functions $G_k^{i,j}$ 
 and $H^{i,j}_k$ for $k,i,j \in Z_{\ge 0}$ such that $ \max\{mi+k, mj+k\} \le \alpha_0+...+\alpha_{m-1}$.
We have the following refinement
 of Le Bruyn-Procesi's result:
\begin{theorem}
\label{genfuncmain}
Let $\lambda \in \c^m$ be generic.
 Then, for $\tau=(-\lambda \cdot \alpha,\lambda)$
 and $\beta=(1,\alpha) \in \Sigma_{\tau}$, 
each of the sets 
$\{G^{i,j}_k\}$ and $\{H^{i,j}_k\}$ forms a generating
set for the ring $\mathcal{O}(\mathfrak{M}^{\tau}(Q_m, \beta))$.
\end{theorem}
\begin{remark}
From now on, for all quiver varieties associated to $Q_m$, we will assume that $\lambda$ is generic.
\end{remark}
The proof of Theorem~\ref{genfuncmain} will be given in the Appendix.

\section{Group action on quiver varieties}
\label{sec3}
In this section we describe a natural subgroup of the automorphism group of 
the free associative algebra on two generators and
define its action on quiver varieties. 

 Recall that the set of trivial paths $\{e_i\}_{i \in I_{\infty}}$ of $Q_m$
form a set of central primitive idempotents, hence 
$S:=\oplus_{i \in I_{\infty}} \c e_i$ is a commutative semisimple algebra. 
Let $\mathrm{Aut}_S(\Pi^{\tau}(Q_m))$ be the group of
$S$-algebra automorphisms of $\Pi^{\tau}(Q_m)$.
By functoriality this group acts on each variety $\mathfrak{M}^{\tau}(Q_m, \beta)$.
In this section, for each integer $m\ge 1$, 
we define the group $G_m$ together with a group homomorphism
to $\mathrm{Aut}_S(\Pi^{\tau}(Q_m))$.
Then $G_m$ acts on $\mathfrak{M}^{\tau}(Q_m, \beta)$ via this homomorphism. 
Note that the case  $m=1$ has been investigated by Berest and Wilson \cite{BW}.
More explicitly, they showed that the action of $G_1$ on each of
the Calogero-Moser space $\mathcal{C}_n$ is transitive.  
Later, in  \cite{BEE}, Berest and two authors of this paper proved 
that $G_1$ acts doubly transitively on  $\mathcal{C}_n$.

Let $R \, = \, \c \langle x, y \rangle$ be the free associative algebra on two generators. 
We denote by $\mathrm{Aut}_{\omega} (R)$ the group of algebra
 automorphisms of $R$ fixing $\omega=xy-yx$. 
Every $\sigma \in \mathrm{Aut}_{\omega} (R)$
 is determined by its action on $x$ and $y$: we will write 
  $(\sigma (x), \, \sigma (y))$ for $\sigma$, where $\sigma (x)$ and $\sigma (y)$ are 
 elements in $R$. 
 We define  $G_m$ to be the subgroup of $\mathrm{Aut}_{\omega} (R)$
 generated by automorphisms
\begin{eqnarray}
\la{x-gr}
\phi_{k,\nu}  &: =& (x, y+ \nu x^{km-1}) \, , \quad \nu \in \c \, , \, k \geq 1,\\
\la{y-gr}
\psi_{k, \mu} &:=&  (x+ \mu y^{km-1}, y)\, , \quad \mu \in \c  \quad k \geq 1.
\end{eqnarray}
%
%
We now discuss a structure of the group $G_m$. First, note for $m\ge 2$, the group $G_m$ is a subgroup of $G_1$.
Second, Makar-Limanov \cite{ML}  showed that $G_1$ is, in fact, equal to $\mathrm{Aut}_{\omega} (R)$.
He also proved  that the projection of $R$ onto $\c[x,y]$ induces an isomorphism of
groups $\mathrm{Aut}_{\omega} (R) \to \mathrm{Aut}_{\omega} (\c[x,y])$, where the latter
is the subgroup of $\Aut(\c[x,y])$ of those automorphisms with Jacobian $1$. Then the well-known
result of Jung and van der Kulk states that $G_1$ can be presented as the {\it amalgamated
free product} $A \ast_U B$, where $A$
is the subgroup of symplectic automorphisms
$$
 (ax+by+e, cx+dy+f)\, , \, a,b,...,f \in \c\, , \, ad-bc=1\, ,
 $$
$B$ is the subgroup of triangular transformations
$$
(ax+q(y),a^{-1}y+h)\, , \, a \in \c^*, \, h \in \c, \,q(y) \in \c[y]
$$
and $U= A \cap B$. For $m \ge 2$, in order to describe the structure of $G_m$ we distinguish two subcases: $m=2$ and $m \geq 3$. 
Let $\Phi$ and $\Psi$ be abelian subgroups generated by automorphisms \eqref{x-gr} and \eqref{y-gr} respectively. Then
%
\begin{proposition} \mbox{}
\label{strGm}
\begin{enumerate}
\item $G_2 = A_1 \ast_{U_1} B_1$ where $A_1 = \SL_2 (\c)$
and $B_1$ is the subgroup of $B$ defined by constraints $q(0)=0$ and $h=0$,
and $U_1= A_1 \cap B_1$. 
\item For $m\ge 3$ we have $G_m \cong \Phi \ast \Psi $.
\end{enumerate}
\end{proposition}
\begin{proof}
\mbox{}
\begin{enumerate}
\item It is an easy consequence of the Jung-van der Kulk theorem.
\item It is sufficient to show that there are no relations between elements of $\Phi$ and  $\Psi$, or equivalently,
$$ \phi_{k_1, \nu_1} \cdot \psi_{l_1,\mu_1}\cdot...\cdot  \phi_{k_n,\nu_n}\cdot \psi_{l_n,\mu_n} \neq 1,$$
which is evident since $m\ge 3$.
\end{enumerate}
\end{proof}
The next result allows us to define an action of $G_m$ on 
preprojective algebras and as a result on the corresponding quiver varieties. 
\begin{lemma}
\label{embGm}
For $\mu \in \c$ and $k \in \z_{\geq 0}$ the maps 
\begin{eqnarray}
&& \widetilde\psi_{k, \mu}  \, : \,  a_i \mapsto  \, a_i   \, +  \, \mu \,  
a_{i-1}^* \ldots a_0^* (a_{m-1}^* \ldots a_0^*)^{k-1} a_{m-1}^* \ldots a_{i+1}^*  \,  , 
\quad a_i^* \mapsto   \, a_i^*  \, ,  e_i \mapsto e_i \, \nonumber \\[0.3cm]
&&
\widetilde\phi_{k, \nu}  \, : \,  a_i \mapsto  \, a_i    \,  , 
\quad a_i^* \mapsto   \, a_i^*   \, +  \, \nu \,  
a_{i+1}\ldots a_{m-1} (a_0 \ldots a_{m-1})^{k-1} a_{0} \ldots a_{i-1} \, ,\,  
e_i \mapsto e_i \, \nonumber 
\end{eqnarray}
define automorphisms of $\Pi^{\tau}(Q_m)$. Furthermore, the map $\rho\, : \, G_m \rightarrow  {\rm Aut}_{S_\infty}(\Pi^{\tau}(Q_m))$ given
by $\rho (\psi_{k, \mu})  = \widetilde\psi_{k, \mu} \, , \, \rho (\phi_{k, \mu})  = \widetilde\phi_{k, \mu}$ defines a homomorphism of groups.
\end{lemma}
\begin{proof}
We first need to prove that $\widetilde \psi_{k, \mu}$ and $\widetilde \phi_{k, \nu} \in  {\rm Aut}_{S_\infty}(\Pi^{\tau}(Q_m))$.
%
 We will check this for  $\widetilde\psi_{k, \mu}$ since similar proof works for $\widetilde\phi_{k, \mu}$. 
Indeed, for $i \in \{1,...,m-1\}$, we have
\begin{eqnarray}
&& \big[a_i +\mu \,  a_{i-1}^* \ldots a_0^* 
(a_{m-1}^* \ldots a_0^*)^{k-1} a_{m-1}^* \ldots a_{i+1}^* \big]\cdot a_{i}^*  \nonumber \\
&-& a_{i-1}^* \cdot \big[a_{i-1} +\mu \,  a_{i-2}^* \ldots a_0^* 
(a_{m-1}^* \ldots a_0^*)^{k-1} a_{m-1}^* \ldots a_{i}^* \big]  \nonumber\\
&=& a_i \cdot a_i^*-a_{i-1}^*a_{i-1}= \lambda_i \, e_i \nonumber
\end{eqnarray}
and for $i=0$, we have
\begin{eqnarray}
&& \big[a_0 +\mu \, (a_{m-1}^* \ldots a_0^*)^{k-1} a_{m-1}^* \ldots a_{1}^* \big]\cdot a_{0}^*  \nonumber \\
&-& a_{m-1}^* \cdot \big[a_{m-1} +\mu \,  a_{m-2}^* \ldots a_0^* 
(a_{m-1}^* \ldots a_0^*)^{k-1}  \big] + a_{\infty}\cdot a_{\infty}^* \nonumber\\
&=& a_i \cdot a_i^*-a_{i-1}^*a_{i-1}+ a_{\infty}\cdot a_{\infty}^*= \lambda_0 \, e_0 \, .\nonumber
\end{eqnarray}
Thus $\widetilde\psi_{k, \mu}$ defines an algebra endomorphism of  $\Pi^{\tau}(Q_m)$.
In order to see that $\widetilde\psi_{k, \mu}$ is an automorphism we note 
that $\rho (\psi_{k, \mu}) \circ \rho (\psi_{k,-\mu})=1$. 
According to Proposition~\ref{strGm} and the universal property of the free product of groups
we conclude that there exists a unique homomorphism with $\rho (\psi_{k, \mu})  = \widetilde\psi_{k, \mu} \, , \, \rho (\phi_{k, \mu})  = \widetilde\phi_{k, \mu}$.
This completes the proof.
\end{proof}
\begin{remark} The map $\c\bar{Q}_m \to \c \bar{{\tilde A}}_{m-1}$ 
that sends $a_{\infty}, a_{\infty}^*$ and $e_{\infty}$ to $0$,
and fixes all other vertices and edges, defines
an algebra epimorphism  $\Pi^{\tau}(Q_m)  \twoheadrightarrow 
\Pi^{\lambda}(\tilde{A}_{m-1})$.  The latter induces the group monomorphism 
$\mathrm{Aut}_{S}(\Pi^{\lambda}(\tilde{A}_{m-1})) \to 
{\rm Aut}_{S_\infty}(\Pi^{\tau}(Q_m))$.
Then it is easy to see that $\rho$ maps $G_m$ to $\mathrm{Aut}_{S}(\Pi^{\lambda}(\tilde{A}_{m-1}))$.
\end{remark}
%

Let $(X_i, Y_i, v, w) \in \mathfrak{M}^{\tau}(Q_m, \beta)$, where $i=0,...,m-1$
which satisfy \eqref{qv1}-\eqref{qv3}. Then, by Lemma~ \ref{embGm},
\begin{eqnarray}
 \psi_{k, \mu}.(X_i, Y_i, v, w)&=& (\tilde{X}_i , Y_i, v, w) \ , \mbox{where}\nonumber\\[0.1cm] 
 \label{ac1}
 \tilde{X}_i &:=& X_i + \mu Y_{i-1}...Y_0 ( Y_{m-1} \ldots Y_0)^{k-1} Y_{m-1} \ldots Y_{i+1}\, ;  \\[0.2cm]
\phi_{k, \nu} .(X_i, Y_i, v, w)&=& (X_i, \tilde{Y}_i,v, w) \, , \mbox{where} \nonumber \\ [0.1cm]
\label{ac2}
\tilde{Y}_i&:=& Y_i+ \nu  X_{i+1}\ldots X_{m-1} (X_0\ldots X_{m-1})^{k-1} X_0\ldots X_{i-1} ;
\end{eqnarray}
define an action of $G_m$ on $\mathfrak{M}^{\tau}(Q_m, \beta)$.
Moreover, this action in terms of quadruples $(X,Y,\bold{v},\bold{w})$ (see \eqref{newdta}
and \eqref{neweq}) is simpler:
\begin{eqnarray}
\psi_{k, \mu}.(X, Y, \bold{v}, \bold{w}) &=&(X+ \mu\, Y^{km-1}, Y,  \bold{v}, \bold{w})\, ,  \\[0.2cm]
\phi_{k,\nu}.(X, Y, \bold{v}, \bold{w}) &=&(X, Y+\nu\, X^{km-1},\bold{v}, \bold{w})\, .
\end{eqnarray}



%
\section{equivariance of reflections}
\label{sec4}

\subsection{Reflection functors}
Let $Q$ be a finite quiver. Let $i$ be a loop-free vertex. Suppose $\tau \in \c^I$ with $\tau_i \neq0$ and $\beta \in \mathbb{N}^I$
such that $\tau \cdot \beta=0$. Then there is an equivalence
functor
\begin{equation}
\label{Ei}
 E_i\,: \, {\rm Rep} (\Pi^{\tau}(Q), \beta) \, \longrightarrow \,
{\rm Rep} (\Pi^{r_i \tau}(Q), s_i (\beta)) \, ,
\end{equation}
where $s_i$ and $r_i$ act as in \eqref{wdim} and \eqref{wpar}.
The reflection functor was defined by Crawley-Boevey and Holland \cite[Theorem 5.1]{CBH} 
and plays an important role in studying deformed 
preprojective algebras.
The functor $E_i$ is $G(\beta)$-equivariant, that is, $G(\beta)$-orbits 
in ${\rm Rep} (\Pi^{\tau}, \beta)$ are mapped to $G(s_i(\beta))$-orbits
in ${\rm Rep} (\Pi^{r_i \tau}, s_i (\beta))$ and hence induces
a bijective map between the corresponding quiver varieties
\begin{equation}
\label{reflecvar}
\mathcal{R}_i: \mathfrak{M}^{\tau}(Q,\beta) \, \xrightarrow{\sim} \, 
\mathfrak{M}^{r_i \tau}(Q, s_i(\beta)) \, .
\end{equation}
This map was independently defined by Nakajima \cite{N} and Lusztig \cite{L}.
Later Maffei showed that, for generic
$\tau$, it is an isomorphism of algebraic varieties (see \cite{M}). We can naturally extend  \eqref{reflecvar} 
to the entire Weyl group. Indeed, for  $s \in W$ such that $s=s_{i_1}...s_{i_k}$,
 define $\mathcal{R}_s: = \, \mathcal{R}_{i_1}\circ ...\circ \mathcal{R}_{i_k}$.
Then by \eqref{reflecvar}, we have an isomorphism
\begin{equation}
\label{reflecvarW}
\mathcal{R}_s: \mathfrak{M}^{\tau}(Q,\beta) \, \xrightarrow{\sim} \, 
\mathfrak{M}^{s \tau}(Q, s(\beta)) \, .
\end{equation}
Now we will explicitly write down the map \eqref{reflecvar} for the quiver
$Q_m$. 
We start  by defining  $\mathcal{R}_i$ for $1\le i \le m-1$.
Let
$$
{\mathbb V}=(V_{\infty},V_0,...,V_{m-1}; X_0,...,Y_{m-1},v,w) \in \mathfrak{M}^{\tau}(Q_m,\beta) \, . 
$$
Let  $\mu: V_i \to V_{i+1} \oplus V_{i-1}$ and $\pi:V_{i+1} \oplus V_{i-1} \to V_i$  be maps given by $\mu(u)=(Y_{i}\, u, -X_{i-1} \,u)$,
  and $\pi(u_1,u_2)=1/\lambda_i (X_i(u_1)+Y_{i-1}(u_2))$ respectively. Using relation \eqref{qv2}, 
  we have $\pi \mu=1$ and hence $\mu\pi$ is an idempotent endomorphism
of $V_{i+1} \oplus V_{i-1}$. Thus ${\mathcal R}_i$ sends ${\mathbb V}$ to ${\mathbb V}'$, where
$V_j'=V_j$ for $j\neq i$ and $V_i'={\rm Im}(1-\mu\pi)$, and  $X_j'=X_j$ and $Y_j'=Y_j$ for $j\neq i-1, i$,
and 
\begin{eqnarray*}
&& X_i'=(-\lambda_i\,{\rm Id} + Y_iX_i, -X_{i-1}X_i)\, , \quad X_{i-1}'(u_1,u_2)=-u_2, \\
&& Y_{i-1}'=(Y_i\,Y_{i-1}, -\lambda_i-X_{i-1}Y_{i-1})\, , \quad Y_i'(u_1,u_2)=u_1,
\end{eqnarray*}
for $(u_1,u_2)\in V_i'$, $v'=v$, $w'=w$.
%

%
%
 
 Similarly for ${\mathcal R}_0$, let $\pi : V_{\infty} \oplus V_1 \oplus V_{m-1} \rightarrow V_0$ 
 and $\mu :  V_0 \rightarrow V_{\infty} \oplus V_1 \oplus V_{m-1}$ be given by
$$
\pi (x,y,z) = \frac{1}{\lambda_0} ( v x+ X_0 y + Y_{m-1} z)\, , \quad
\mu ( t ) =  (wt, Y_0\, t,  -X_{m-1}t ) ,
$$
%
%
and
hence by \eqref{qv1} we obtain $\pi\mu=1$, which implies that $\mu\pi$ is an idempotent 
endomorphism of $V_{\infty} \oplus V_1 \oplus V_{m-1} $.
 Therefore ${\mathcal R}_0$ is a map from ${\mathbb V}$ to ${\mathbb V}'$, where
$V_j'=V_j$ for $j\neq 0$ and $V_0'={\rm Im}(1-\mu\pi)$, and  $X_j'=X_j$ and $Y_j'=Y_j$ for $j\neq m-1, 0$,
and 
%
%
%
%
\begin{eqnarray}
\label{ref0x}
&& X_0'  = (w X_0,  -\lambda_0 {\rm Id}_{n_1} + Y_0 X_0 ,  -X_{m-1} X_0)\, , \,
 X_{m-1}'(u_1,u_2,u_3)   =  -u_3,   \\
 \label{ref0y}
&& Y_{m-1}' =  (w Y_{m-1} , Y_0 Y_{m-1}, -\lambda_0 {\rm Id}_{n_{m-1}} -X_{m-1} Y_{m-1} )\,, \,
Y_0' (u_1,u_2,u_3)=  u_2,\\
\label{ref0uv}
&&  v' = ( -\lambda_0 + w v, Y_0 v ,  -X_{m-1} v)\, , \,  
w' (u_1,u_2,u_3)   = u_1,
\end{eqnarray}
for $(u_1,u_2,u_3) \in V_0'$.


Our first application of the reflection functors is to show that a quiver variety
 associated to $Q_m$ 
 is isomorphic to the unique Calogero-Moser space 
 corresponding to the same quiver. To this end, we need the following 
 proposition
%
\begin{proposition}
\label{propbij}
 Let $\beta=(1,\alpha) \in \Sigma_{\tau}$, where 
 $\alpha=(\alpha_0,...,\alpha_{m-1}) \in \mathbb{N}^m$ . Then there is $w \in W_{\infty}$
 so that $w(\beta)=(1,{\bf n}) \in \mathbb{N}^{m+1}$, where ${\bf n}=(n,...,n)\in {\mathbb N}^{m}$.
\end{proposition}
\begin{proof}
The proof for $m=2$ is straightforward. Suppose $m>2$, and put $\mu={\rm max} \{\alpha_0,...,\alpha_{m-1}\}$.
We recall that the action of $W_{\infty}$ on  $\z^m$ is given by \eqref{wact1}-\eqref{wact2}.
%
%
Let $k$ be the smallest index such that $\alpha_k= \mu$. We claim that
$\alpha_{k+1}+\alpha_{k-1} \ge \alpha_k$ if $k>0$, and
$\alpha_1+\alpha_{m-1}+1 \ge \alpha_0$ if $k=0$. Assuming this claim, we have
$s_k(\beta)=(1,\alpha')$ so that $\alpha_k'=\alpha_{k+1}+\alpha_{k-1}-\alpha_k$
and $\alpha'_i=\alpha_i$ for $i\neq k$. In particular, $\alpha' \in \mathbb{N}^m$ and 
$\alpha'_k <\alpha'_{k+1}$.
The latter implies that if $\mu':={\rm max}\{\alpha_0',...,\alpha_{m-1}'\}$
and $k'$ is the smallest index with $\alpha_{k'}=\mu'$
then $k<k'$.  By repeating this procedure sufficiently many times
 yields $\alpha'$ are the same.

Now we proceed to prove the above claim. Assume on the contrary, for $k>0$,
$\alpha_k-\alpha_{k-1}-\alpha_{k+1}>0$.
Then, by Lemma~\ref{lemmamain}$(ii)$ we have
 \begin{eqnarray*}
 \mathrm{dim}_{\c} \, \mathfrak{M}^{\tau}(\beta) &\le & 2 \alpha_0
 -(\alpha_{k-1}-\alpha_k)^2-(\alpha_k-\alpha_{k+1})^2\\
 &\le & 2 \alpha_0 - 2\alpha_k(\alpha_k-\alpha_{k-1}-\alpha_{k+1})
 \le 2\alpha_0 - 2\alpha_k <0,
\end{eqnarray*}
which is impossible. Hence $\alpha_k-\alpha_{k-1}-\alpha_{k+1}\le 0$.

Similarly, for $k=0$ assuming $\alpha_0-(\alpha_1+\alpha_{m-1}+1)>0$, we obtain
\begin{eqnarray*}
\mathrm{dim}_{\c} \, \mathfrak{M}^{\tau}(\beta) &\le & 
2 \alpha_0-(\alpha_{m-1}-\alpha_0)^2-(\alpha_0-\alpha_1)^2\\
&\le &2 \alpha_0-2 \alpha_0(\alpha_0-\alpha_1-\alpha_{m-1})
\le 2\alpha_0-4\alpha_0=-2\alpha_0<0,
\end{eqnarray*}
and
again we get contradiction, which proves our claim.
\end{proof}
%
%
%
Combining  Proposition~\ref{propbij} with the above mentioned result of Maffei (see \eqref{reflecvarW}), we obtain
\begin{corollary}
\label{red}
For $\beta =(1, \alpha) \in \Sigma_{\tau}$, let $s \in W$ be such that $s(\beta)=(1,{\bf n})$.
Then there is an isomorphism of algebraic varieties
$$ 
\mathcal{R}_s: \mathfrak{M}^\tau(Q_m, \beta) \to {\mathcal C}_{n}^{s(\tau)}(Q_m) .
$$
\end{corollary}
On the ring of functions, the reflection $\mathcal{R}_s$ induces an algebra isomorphism
$$ \mathcal{R}_s^*\, : \, \mathcal{O}\big(\mathfrak{M}^{s(\tau)}(Q_m, s(\beta))\big) \to
\mathcal{O}\big(\mathfrak{M}^{\tau}(Q_m,\beta) \big) \, .$$
Next we describe the image of the generators $H_{k}^{i,j}$ under $ \mathcal{R}_l^*$ for $l \in \z / m \z$ (see Theorem~\ref{genfuncmain} and \eqref{GHgen}). 
 Let $(X_i,Y_i,v,w) \in \mathfrak{M}^{\tau}(Q_m,\beta)$. Let
$(X'_i,Y'_i,v',w')=\mathcal{R}_l (X_i,Y_i,v,w)$, and let  $H_{k}^{'i,j} = R_l ^* (H_{k}^{i,j})$.
The following will be essential for proving the equivariance of the reflection functors.
\begin{lemma}
\label{lemmaH}
For  $k \neq l$, $ H_{k}^{'i,j}= H_{k}^{i,j}$ and  $H_{l}^{'i,j}=H_{l}^{i,j}-\lambda_l \, H_{l-1}^{i-1,j-1}$.
\end{lemma} 
\begin{proof}
We will give the proof only for the case $l=0$ since remaining cases 
can be checked in the similar manner. First, since $\mathcal{R}_0$ changes maps that originate 
or terminate at the $0$ vertex, we have $X'_i=X_i$ and $Y'_i=Y_i$ for all $i \neq 0,m-1$. Moreover
using \eqref{ref0x} - \eqref{ref0uv} yields
\begin{eqnarray*}
 X'_{m-1}\, X'_0\, =\, X_{m-1}\, X_0\, , &&
Y'_0\, Y'_{m-1}\, =\, Y_0\,Y_{m-1} \, , \\ [0.1cm]
w' \, Y'_{m-1}\, =\, w\,Y_{m-1}\, , && X_{m-1}'\, v'\,=\, X_{m-1}\,v\,,
\end{eqnarray*}
and therefore, for $1\le k \le m-1$, we immediately obtain $H_k^{'i,j}=H^{i,j}_{k}$.
%
%
For $k=0$, using the above identities and  
$Y_{0}'X_0'=-\lambda_0\mathrm{Id}_{n_1}+Y_0X_0$ we have
\begin{eqnarray*}
\label{defbijk}
H_0^{'i,j}&:=&w' \big[ (Y'_{m-1} \,Y_{m-2} \ldots Y_1\, Y'_0)^i\,
(X'_0 \, X_1 \ldots X_{m-2}\, X'_{m-1})^j \big] v'  \\[0.2cm]
&=&w \big[ A^{i-1} (Y_{m-1} \ldots Y_1) \, (Y_{0}'X_0')\,(X_1\ldots X_{m-1})\,B^{j-1} \big] v
\\[0.2cm]
&=&H_{0}^{i,j}-\lambda_0\, H_{m-1}^{i-1,j-1}\,.
\end{eqnarray*}
\end{proof}
\subsection{$G_m$-equivariance }
First, we recall that there is a well-defined action of the group $G_m$ on   $\mathfrak{M}^{\tau}(Q_m, \beta)$
given by  \eqref{ac1}-\eqref{ac2}. Second, there is an isomorphism \eqref{reflecvarW} between quiver varieties 
induced by reflection functors. Our goal in this section is to prove
\begin{theorem}
For any $s \in W$ and any $\sigma \in G_m$, $\mathcal{R}_s \circ \sigma =\sigma \circ \mathcal{R}_s$. That is, the
following diagram
$$
 \begin{tikzcd}
   \mathfrak{M}^{\tau}(Q,\beta) \arrow[d, "\sigma"] \arrow[r, "\mathcal{R}_s"]&
   \mathfrak{M}^{s \tau}(Q, s(\beta))\arrow[d, "\sigma"]\\
    \mathfrak{M}^{\tau}(Q,\beta) \arrow[r, "\mathcal{R}_s"]&\mathfrak{M}^{s \tau}(Q, s(\beta))
    \end{tikzcd}
   $$
   commutes.
\end{theorem}
\begin{proof}
It sufficies to prove the statement only for generators of $W$ and $G_m$.
Furthermore, we will only check it for $s=s_0$ and $\sigma=\psi_{k,\mu}$, 
since verification of the
other cases is similar. Let $(X_i,Y_i,v,w) \in \mathfrak{M}^{\tau}(Q,\beta)$, and let
\begin{eqnarray*}
(\bar{X}_i,\bar{Y}_i,\bar{v},\bar{w})&:=&\mathcal{R}_0 \,\circ \,  \psi_{k,\mu} (X_i, Y_i, v, w)\, , \\
(\hat{X}_i, \hat{Y}_i,\hat{v},\hat{w})&:=&\psi_{k,\mu}  \,\circ \, \mathcal{R}_0  (X_i, Y_i, v, w) \, .
\end{eqnarray*}
Let be $\bar{H}_k^{i,j}$ and $\hat{H}_k^{i,j}$ be the values of $H^{i,j}_k$ function at respective
points. Then, by Theorem~\ref{genfuncmain}, it suffices to show $\bar{H}_k^{i,j}=\hat{H}_k^{i,j}$
for all $i,j$ and $k$. If $\tilde{H}^{i,j}_k$  is the value of $H^{i,j}_k$ at 
$(\tilde{X}_i,Y_i,v,w)=\psi_{k,\mu} (X_i, Y_i, v, w)$ then, by Lemma~\ref{lemmaH},
\begin{equation*}
 \bar{H}_0^{i,j} =  \tilde{H}^{i,j}_0- \lambda_0 \, \tilde{H}^{i-1,j-1}_{m-1}  \, , \quad
  \bar{H}_k^{i,j} = \tilde{H}^{i,j}_k , \quad k = 1, \ldots, m-1.
 \end{equation*}
On the other hand, by the definitions of $\psi_{k,\mu}$ and $\mathcal{R}_0$, we have
$$ \hat{X}_i=  X_i' + \mu Y_{i-1}'...Y_0' ( Y_{m-1}' \ldots Y_0')^{k-1} Y_{m-1}' \ldots Y_{i+1}'
\, , \,  \hat{Y}_i=Y_i'\, ,\, \hat{v}=v'\, , \,  \hat{w}=w'\, . $$
Since $Y_i'=Y_i$ for $1\le i \le m-2$, we obtain $\hat{X}_i=\tilde{X}_i$. Next,
since $Y_{0}'X_0'=-\lambda_0\mathrm{Id}_{n_1}+Y_0X_0$ and  $Y_0'  Y_{m-1}' = Y_0 Y_{m-1}$, it follows
\begin{eqnarray*}
\hat {Y}_0 \hat{X}_0 &= & \,Y_0' X_0' + \mu \, Y_0' \Big [ ( Y_{m-1}' \ldots Y_{0}')^{k-1} Y_{m-1}' \ldots  Y_{1}'  \Big]  
 \\
&= &  -\lambda_0\mathrm{Id}_{n_1} + Y_0  X_0 + \mu \,   
Y_0 \Big [ ( Y_{m-1} \ldots  Y_{0})^{k-1} Y_{m-1} \ldots Y_{1}  \big]\\
&=& -\lambda_0\mathrm{Id}_{n_1} + Y_0 \tilde{X}_0 \, .
\end{eqnarray*}
Furthermore
\begin{eqnarray*}
\hat{X}_{m-1} \hat{X}_0 & = &  X_{m-1}' X_0' +  \mu \, X_{m-1}' \, ( Y_{m-1}' \ldots Y_{0}')^{k-1}  
Y_{m-1}'  \ldots  Y_{1}' \\
&+& \mu \, (  Y_{m-2}' \ldots   Y_{0}'  Y_{m-1}')^{k-1}  Y_{m-2}' \ldots  Y_{0}' \,  X_0'   \\
 &  +  &    \mu^2 \, ( Y_{m-2}' \ldots  Y_{0}' Y_{m-1}')^{k-1}   Y_{m-2}' \ldots  Y_{0}' \, 
 (Y_{m-1}' \ldots  Y_{0}')^{k-1}  Y_{m-1}' \ldots   Y_{1}'  \\
  & = &X_{m-1}'  X_0' +  \mu \, ( X_{m-1}' \, Y_{m-1}') \, ( Y_{m-2}'  \ldots   Y_{0}')^{k-1} Y_{m-1}'  \ldots  Y_{1}'  \\
  & + & \mu \, (  Y_{m-2}' \ldots  Y_{0}'  Y_{m-1}')^{k-1}  Y_{m-2}' \ldots   Y_1' \, (Y_{0}' \,  X_0')   \\
 &  +  &    \mu^2 \, ( Y_{m-2}' \ldots  Y_{0}' Y_{m-1}')^{k-1} Y_{m-2}' \ldots   Y_{0}' \, ( Y_{m-1}' \ldots  Y_{0}')^{k-1} 
Y_{m-1}' \ldots  Y_{1}'\, .
  \end{eqnarray*}    
Now using $ Y_{i}' = Y_i$ for  $1\le i \le m-2 $, $X_{m-1}' X_0' =X_{m-1} X_0 \, , \, Y_{0}' Y_{m-1}' = Y_0 Y_{m-1}$ and $X_{m-1}' \,  Y_{m-1}' = X_{m-1} \, Y_{m-1} +\lambda_0\mathrm{Id}_{n_1} \, , \,   
Y_{0}' \, X_0' =  -\lambda_0\mathrm{Id}_{n_1} + Y_0 X_0 $, we obtain $\hat{X}_{m-1} \hat{X}_0
=\tilde{X}_{m-1} \tilde{X}_0$.

By using $ X_{m-1}'  v' = X_{m-1} v $ and $Y_0'  v'  = Y_0 v$, we have 
\begin{eqnarray*}
\hat{X}_{m-1} \hat{v} &=&  X_{m-1}'  v' +   \mu \, ( Y_{m-2}' \ldots   Y_{0}' Y_{m-1}')^{k-1} Y_{m-2}' \ldots  Y_{0}' v'\\
  & = &  X_{m-1}  v +   \mu \, (   Y_{m-2} \ldots   Y_{0}  Y_{m-1})^{k-1}   Y_{m-2} \ldots   Y_{0}  v = \tilde{X}_{m-1} v
\end{eqnarray*}
Combining the above identities, we get
\begin{eqnarray*}
\hat{H}^{i,j}_0& = &  \hat w  ( \hat{Y}_{m-1}' \ldots \hat{Y}_1 \hat{Y}_0)^{i-1} \hat{Y}_{m-1} \ldots \hat{Y}_1 (\hat{Y}_0  \, \hat{X}_0) \hat{X}_1 \ldots \hat{X}_{m-1} \, (\hat{X}_0\hat{X}_1 \ldots \hat{X}_{m-1})^{j-1} \hat{v} \\
& = &   w  ( Y_{m-1} \ldots  Y_1  Y_0)^{i-1}  Y_{m-1} \ldots  Y_1 ( Y_0  \, \tilde{X}_0)   \tilde{X}_1 \ldots  \tilde{X}_{m-1} \, ( \tilde{X}_0  \tilde{X}_1 \ldots  \tilde{X}_{m-1})^{j-1}  v\\
& - &  \lambda_0 w  ( Y_{m-1} \ldots  Y_1  Y_0)^{i-1}  Y_{m-1} \ldots  Y_1 \tilde{X}_1 \ldots  \tilde{X}_{m-1} \, ( \tilde{X}_0  \tilde{X}_1 \ldots  \tilde{X}_{m-1})^{j-1}  v\\
& = &\tilde{H}^{i,j}_0- \lambda_0 \, \tilde{H}^{i-1,j-1}_{m-1} 
\end{eqnarray*}
and hence $\hat{H}^{i,j}_0=\bar{H}^{i,j}_0$. Similarly, we can verify 
$\hat{H}^{i,j}_{k}=\bar{H}^{i,j}_{k}$ for  $1 \le k \le m-1$. 
\end{proof}
\begin{remark}
Thus the proof of the $G_m$-transitivity on
general quiver varieties reduces to that on the generalized Calogero-Moser spaces.  
\end{remark}

\section{Proof of Theorem~\ref{maintheorem}}
\label{sec5}

First, we review some basic facts about rational Cherednik algebras. Then we show for 
$W= S_n \rtimes \Gamma$ there is a natural action of $G_m$ on $X_c$, the variety defined by the center of a 
Cherednik algebra. We prove that this action is transitive. At the end of the section, we prove the equivariant 
version of the Etingof-Gizburg isomorphism between $X_c$ and 
 generalized Calogero-Moser space $\CC^{\tau}_n(Q_m)$. 
 Therefore by the reduction result (see Corollary \ref{red}), we conclude that $G_m$ acts transitively on all
 Nakajima varieties $\mathfrak{M}^\tau(Q_m, \beta)$.
 
\subsection{Hamiltonian flows and transitivity}

An algebraic (one-parameter) flow on an algebraic variety $X$ is a morphism
 of varieties $\phi : \c \times X \rightarrow X$
satisfying two conditions:
\begin{enumerate}
\item $\phi (0, p) = p$ for all $p \in X$;
\item  $\phi(t_1 + t_2, p) = \phi (t_1 , \phi(t_2, p))$.
\end{enumerate}
Equivalently, this gives rise to a one-parameter subgroup $\{ \phi_t \}_{t \in \c} \subset \Aut (X)$; with such 
flow we associate a derivation $\partial : \mathcal{O} (X) \rightarrow \mathcal{O}(X)$
$$
\partial (f) = \frac{d}{dt} |_{t=0} \phi_t^* (f).
$$
Suppose that $X$ is a symplectic variety with 
symplectic form $\omega$. Then $\mathcal{O}(X)$ is a Poisson algebra and there is a bijection 
$$
\omega^{\#} \, : \,  {\rm Der} (\mathcal{O}(X) )  \rightarrow \Omega^1 (X)\,.
$$ 
If  a flow $\{ \phi_t \}$ is Hamiltonian, i.e. $\phi_t ^* \omega = \omega$  for all $t \in \c$, then $\omega^{\#} (\partial )$ is an exact 1-form, that is, we can find $f \in \mathcal{O}(X)$ such that $\omega^{\#} (\partial ) = df$. Using this description we can express the derivation
in terms of the Poisson bracket
$$
\partial  = \{ f,  \, \}.
$$

\begin{lemma}
\la{flows}
Let $G$ be a subgroup of $\Aut(X)$ generated by a family of Hamiltonian flows $\{\phi_i (t)\}$ for $i \in \mathcal{I}$, with Hamiltonians
$f_i \in \mathcal{O}(X)$. Moreover, suppose $\{f_i\}_{i \in \mathcal{I}}$  Poisson generate $\mathcal{O}(X)$. Then $G$ acts transitively
on $X$.
\end{lemma}
\begin{proof}
By \cite{AFKKZ}, we know that the orbits of $G$ on $X$ are locally closed. Thus it 
suffices to show that there are no proper $G$-invariant closed subsets of $X.$
Let $ I$ be the (reduced) defining ideal of $Y.$ Infinitesimally this means that $ \lbrace f_i , I\rbrace\subset I$ 
for all $ i \in \mathcal{I}.$ Since $\{ f_i \}_{i \in \mathcal{I}}$ Poisson generates $\mathcal{O}(X),$ it follows that $I$ is a proper 
nonzero Poisson ideal of  $\mathcal{O}(X)$. This is a contradiction to the symplecticity of $X$.
\end{proof}

\subsection{Preliminaries on rational Cherednik algebras}
\ni Let $W$ be a complex reflection group; $\mathfrak{h}$ its reflection representation
and $S\subset W$  the set of all complex reflections.
Let $(\,,\,):\mathfrak{h}\times \mathfrak{h}^*\to \mathbb{C}$ be the natural pairing. Given a reflection $s\in S,$
let $\alpha_s\in \mathfrak{h}^*$ be an eigenvector of $s$ for eigenvalue $1$.
Also, let $\alpha_s^{\vee} \in \h$ be an eigenvector normalized so that $\alpha_s(\alpha_s^{\vee} )=2.$
Let $c:S\to \mathbb{C}$ be a function invariant with respect to conjugation with $W$ and $t\in\mathbb{C}.$
The rational Cherednik algebra $H_{t,c}(W, \mathfrak{h})$ associated to $(W, \mathfrak{h})$ with parameters
$(t, c)$ is defined as the quotient of $\mathbb{C}[W]\ltimes T( \h\oplus \h^*)$ by the following relations
$$
[x, y]=t(y,x)-\sum_{s\in S}c(s)(y,\alpha_s)(\alpha_s^{\vee} , x),\quad [x, x']=0=[y,y']
$$
for all $x, x'\in \mathfrak{h}^*$ and $y, y'\in \mathfrak{h}.$
Recall that there is a natural grading on this algebra defined as follows:
$$
\deg y=-1, y\in \mathfrak{h}; \deg(w)=0, w\in W; \deg(x)=1, x\in \mathfrak{h}^*.
$$
A key result for rational Cherednik algebras is that they have the PBW property (\cite[Theorem 1.3]{EG}),
which states that there is a vector space isomorphism
$$
H_{t, c}(W, \mathfrak{h}) \cong \c[\h] \otimes \c W \otimes \c[\h^*]
$$
for all values of $t$ and $c$.  For $t=0$, the algebra $H_c (W, \mathfrak{h}) : =H_{0, c}(W, \mathfrak{h})$ has a large center $Z_c(W, \mathfrak{h})$ 
which is an affine domain. Let $X_{c} (W) = {\rm Spec} (Z_c(W, \mathfrak{h}))$ be the corresponding affine variety. The inclusions
$\c[\h]^W \hookrightarrow Z_c(W, \mathfrak{h})$  and $\c[\h^*]^W \hookrightarrow Z_c(W, \mathfrak{h})$ define surjective
homomorphisms
$$
\pi_1 : X_c (W) \twoheadrightarrow  \h^* / W  \quad   \mbox{and} \quad \pi_2 : X_c (W)  \twoheadrightarrow \h / W
$$
The product morphism 
$$
\Gamma : X_c (W)   \twoheadrightarrow  \h^* / W  \times \h /W
$$
is a finite, closed, surjective morphism.
Viewing $H_{t,c}(W, \h)$ as an algebra over $\mathbb{C}[t],$ we have a canonical isomorphism
$$
H_{c} (W, \mathfrak{h}) \, \cong \, H_{t,c}(W. \h) / tH_{t,c}(W, \h),
$$
which defines a natural Poisson bracket on $Z_c(W, \h)$. See \cite{B} for details. 
For the convenience of the reader we review a description of the completion of rational Cherednik algebras
due to Bezrukavnikov and Etingof. 
Let $b\in \mathfrak{h}$ and let $\m$ be the maximal ideal of $\mathbb{C}[\mathfrak{h}]^{W}$ corresponding 
to the orbit $W \cdot b.$ Denote by $\widehat{H_{t,c}(W, \mathfrak{h})}_b=H_{t,c}\otimes_{\mathbb{C}[h]^{W}}\widehat{\mathbb{C}[\h]^{W}}_{\m}$,
the completion of $H_c(W, \h)$ with respect to $\m.$ Denote by $W_b$ the stabilizer of $b$. It is a {\it parabolic 
subgroup}. By Steinberg's Theorem $W_b$ is also a complex reflection group.  
Then we have $\h=\h^{W_b}\oplus ((\h^*)^{W_b})^{\perp}$, a decomposition of $\h$ as $W_b$-module. 
Denote by $c'$ the restriction of $c$ on reflections in $W_b.$
This allows us to define $H_{c'}(W_b, \h).$

Next we recall the centralizer algebra construction. 
Let $K$ be a finite group and $H$ its subgroup. Let $A$ be an algebra containing
the group algebra $\c H$. Let $P= \text{Fun}_H(K, A )$ be the set of all 
$H$-invariant functions: $f: K \to A$  such that $f(hg)=hf(g)$ for 
$h\in H,g \in K.$ Then $P$ is a free right $A$-module of rank $|K/H|$. 
Denote by $C(K,H,A)$ the centralizer algebra $End_A(P)$ of the module $P$.
By choosing left coset representatives of $H$ in $G$, $C(K,H,A)$ 
can be realized as the matrix algebra of size $|K/H|$ over $A$.
Now we are ready to state result of Bezrukavnikov-Etingof:
\begin{theorem}[\cite{BE}, Theorem 3.2]
There is an isomorphism of algebras
$$
\theta:\widehat{H_{t,c}(W, \h)}_b\cong C(W, W_b,\widehat{H_{t,c'}(W_b, \h)}_0),
$$
defined by following formulas. Suppose that $f\in \text{Fun}_{W_b}(W, \widehat{H_{t,c'}(W_b,\h)}_0)$. Then
$$
\theta(u)(f)(w)=f(wu) \quad  \mbox{for}  \quad u \in W;
$$ 
for any $\alpha \in \h^*$,  define
\begin{equation}
\la{thetax}
\theta(x_{\alpha})(f)(w)=(x_{w\alpha}^{(b)}+(w \alpha,b))f(w)
\end{equation}
where $x_{\alpha} \in  \h^* \subset H_{t,c}(W, \h)$ and $x_{w\alpha}^{(b)} \in H_{t,c'}(W_b,\h) \, $; and finally for $a \in \h$, 
\begin{equation}
\la{thetay}
\theta(y_a)(f)(w)=y_{w a}^{(b)} f(w)+\sum_{s\in S: \, s \notin W_b}\frac{2c(s)}{1-\lambda_s}\frac{\alpha_s(w a)}{ x_{\alpha_s}^{(b)}+\alpha_s(b)}(f(sw)-f(w)),
\end{equation}
where $y_a \in \h \subset H_{t,c}(W, \h)$ and $y_{w \alpha}^{(b)} ,  \, x_{\alpha_s}^{(b)} \in H_{t,c'}(W_b,\h)$. Here we denote by $z^{(b)}$ an element $z \in H_{t,c}(W, \h)$
but considered in $H_{t,c'}(W_b,\h)$.
\end{theorem}
\ni Moreover for $t=0$, when restricted to centers we obtain an isomorphism of Poisson algebras
\begin{equation}
\la{theta}
\theta:\widehat{Z_c(W, \mathfrak{h})}_b\cong \widehat{Z_{c'}(W_b, \mathfrak{h})}_0.
\end{equation}
\textbf{From now on, we will be assuming that $t=0.$}
First we need to describe images of the subalgebras $\mathbb{C}[\h]^{W}, \mathbb{C}[\h^*]^{W}$ under $\theta .$
It follows from \eqref{thetax} that 
$$
\theta g(x)=g(x+b), x\in \h, g\in \mathbb{C}[\h]^W.
$$
Moreover, the image $\theta(\c[\h]^{W})$ is dense in $\widehat{\c[\h]^{W_b}_0}$ by \cite[Lemma 4.4]{B}.
Notice that even though the algebra $\widehat{H_{c'}(W_b,\h)}_0$ is no longer graded,
we still have the filtration on it corresponding to the grading on 
$H_{c'}(W_b,\h)$. This filtration is compatible with the Poisson bracket and 
$\gr{\widehat{H_{c'}(W_b,\h)}_0}=H_{c'}(W_b, \mathfrak{h})$.
We have the induced filtration on $\widehat{Z_{c'}(W_b, \mathfrak{\h})}_0.$
Now it follows easily from \eqref{thetay} that
$$
\mathrm{gr}
(\theta |_{\c[\h^*]^{W}}) 
= \mathrm{Id}_{\c[\h^*]^{W}}.
$$ 
Now we can easily show the following.
\begin{lemma}\label{reduction}
Let $\mathcal{O}\subset \mathbb{C}[\h^*]^{W}$. If  $\mathcal{O}$ and $ \, \mathbb{C}[\h]^{W_b}$
 generate $Z_{c'}(W_b, \h)$ as a Poisson algebra, then the resulting Poisson algebra 
 is dense in $\widehat{Z_c(W, \h)}_b$.
\end{lemma}

\begin{proof}
In view of the Poisson isomorphism \eqref{theta}, we need
to show that $\theta(\mathcal{O}), \theta(\mathbb{C}[\h^*]^{W})$
generate a dense subalgebra of $\widehat{Z_{c'}(W_b, \mathfrak{h})}.$
By remarks preceding this Lemma, $\theta(\mathbb{C}[\h]^{W})$ is dense in $\widehat{\mathbb{C}[\h]^{W_b}_0},$
and $\gr{\theta|_{\mathbb{C}[\h^*]^{W}}}=\mathbb{C}[\h^*]^{W}.$ 
Since by assumption $\mathcal{O}$ and $\mathbb{C}[\h]^{W_b}$ Poisson
generate $Z_{c'}(W_b, \h),$ this completes the proof.
\end{proof}

Next we will use the following result.

\begin{lemma}\label{key}
Let $Z$ be a  Poisson algebra over $\mathbb{C}$
such that $\spec Z$ is a symplectic variety.
Let $I\subset Z$ be a finite codimensional proper involutive ideal: $ \lbrace I, I\rbrace\subset I.$
Then $I\subset \m^2$ for every maximal ideal $\m$ containing $I.$
\end{lemma}

\begin{proof}
Suppose that the conclusion does not hold. Then there exists
 $f\in I$ and a maximal ideal $\m$ containing $I$ such that $f\notin \m^2.$
 Since $Z/I$ is finite dimensional, there is some integer $l>0$ that
$\m^l\subset I$. Since $Z$ is symplectic, it follows from the formal Darboux theorem that
$\hat{Z}_{\m}=\mathbb{C}[[f_1,\cdots, f_n, g_1,\cdots, g_n]]$ with the usual Poisson bracket $\lbrace f_i,g_j\rbrace=\delta_{ij}.$
Moreover, we may assume without loss of generality that $f_1=f$. Clearly $I\hat{Z}_{\m}$ is a proper involutive ideal of $\hat{Z}_{\m}.$
Since $g_1^l\in I\hat{Z}_{\m},$ we get that $1\in I\hat{Z}_{\m},$  a contradiction.
\end{proof}

In what follows, we will only deal with Cherednik algebras and Calogero-Moser spaces 
associated with the wreath product of $S_n$ with a cyclic group $\Gamma$ of order $m$. 
 Let $W=S_n\ltimes \Gamma^n$ which we will simply denote by $S_n\wr \Gamma$. 
 Given $\gamma \in \Gamma$, we write $\gamma_i $ for $\gamma$ placed
 in the $i$-th factor of $\Gamma^n$.  Let $\alpha$ be a fixed 
generator of $\Gamma$ and $s_{ij}$ be an elementary transposition $i \leftrightarrow j$ in $S_n$.
The group $W$ acts naturally on $\mathfrak{h}=\mathbb{C}^n$
with basis vectors $y_1,\cdots,y_n : \,$ (1) $\, S_n$ acts by permutations on $\mathbb{C}^n$ ; (2) 
$\Gamma^n \times \mathbb{C}^n \rightarrow  \mathbb{C}^n$ is simply the $n$-th Cartesian power of the
natural cyclic group action $\Gamma \times \c \rightarrow \c :  \, \alpha \cdot z = \xi z$ where $\xi$ is
the $m$-th primitive root of unity.

Let $x_1,\cdots,x_n$ be the dual basis of $\mathfrak{h}^*.$
Let $ c=(c_0,\cdots, c_m)\in \mathbb{C}^{m+1}$ be a set of parameters. 
Then the corresponding
Cherednik algebra $H_c(W, \mathfrak{h})$ is defined as the quotient of
$\mathbb{C}[W]\ltimes T(\mathfrak{h}\oplus\mathfrak{h}^*)$ by the  relations
\begin{eqnarray*}
\left[ x_i, x_j\right] & = &\left[y_i, y_j\right]=0, \\
\left[ y_i, x_j \right] & = & -c_0\sum_k s_{ij}\xi^k\alpha_i^k\alpha_j^{-k},\\
\left[ y_i,x_i \right] & = &
c_0\sum_{j\neq i}\sum_{k=0}^{m-1}s_{ij}\alpha_i^k\alpha_j^{-k}+
\sum_{k=1}^{l-1}c_k\alpha_i^k\, .
\end{eqnarray*}
It is easy to see that
$$
\sum_{i=1}^n x_i^{km}, \quad
\sum_{i=1}^ny_i^{km}\in Z_c(W, \mathfrak{h}) \, , \quad k\geq 1.
$$
To shorten the notation, we omit $(W, \h)$ and simply write $H_c$ and $Z_c$.
Recall $X_c =\spec Z_c$. 
\begin{theorem}\label{main}
%
 Let $c$ be  such that $X_c$ is smooth. 
Let $\mathcal{O}$ be a graded subalgebra of $\mathbb{C}[\h]^{W}$
 such that $\mathbb{C}[\h]^{W}$ is finite over $\mathcal{O}.$  Then $Z_c$ is generated as a Poisson algebra 
by $\mathcal{O}$ and  $\mathbb{C}[\h^*]^{W}.$
Similar statement holds for a graded subalgebra $\mathcal{O}'\subset \mathbb{C}[\h^*]^{W}$
such that $ \mathbb{C}[\h^*]^{W}$ is finite over $\mathcal{O}'.$


\end{theorem}

We will use the following result kindly communicated to us by G.Bellamy.
The proof is due to him, all possible mistakes are ours.
\begin{lemma}
\label{GB}
Let  $W$ be a parabolic subgroup for the wreath product $S_n\wr \Gamma$.
Then there exists $f\in (\mathbb{C}[\mathfrak{h}]^W)_+$  such that
for any maximal ideal $\m \subset Z_c(W, \mathfrak{h})$ containing 
$(\mathbb{C}[\mathfrak{h}]^W)\otimes \mathbb{C}[\mathfrak{h}]^W)_+$ and
lying in the smooth locus of $X_c(W),$
we have $f\notin \m^2.$
%
%
\end{lemma}

\begin{proof}
Since parabolic subgroups of $S_n\wr \Gamma$ consist of direct products of similar wreath products,
we may assume without loss of generality that $W=S_n\wr \Gamma.$
Then we claim that elements $f=\sum_{i=1}^nx_i^{m}, g=\sum_{i=1}^ny_i^{m}$ do not belong to $\m^2$ for any
maximal ideal $\m \subset Z_c(W, \mathfrak{h})$ containing $(\mathbb{C}[\h]^{W}\otimes \mathbb{C}[\h^*]^{W})_+ \, $.
We prove this by contradiction.

Let $\m$ be a maximal ideal as above containing $f.$
As explained in \cite[Section 8.7]{BT}, $\m/\m^2$ is a graded vector space under the
$\mathbb{C}^*$-action on $Z_c(W, \mathfrak{h})$ such that
$\m/ \m^2=V^+\oplus V^{-}$ where $V^+$ and $V^{-}$ are 
spanned by positive and negative degree basis vectors respectively with
 $\dim V^+=\dim V^-=n.$ Let  $p_1,\cdots,p_n\in \m$ (resp. $q_1,\cdots, q_n \in \m$) be 
 homogeneous lifts of basis of $V^{+}$ (resp. $V^{-}$).
As explained in \cite{BT}, we may further assume that $\mathbb{C}[\h]^{W}\subset \mathbb{C}[p_1,\cdots,p_n]$ and 
$\mathbb{C}[\h^*]^{W}\subset \mathbb{C}[q_1,\cdots, q_n]$.  Suppose that $f\in  \m^2.$
Then by simple examination we obtain that $f\in (V^{+})^2\mathbb{C}[p_1,\cdots,p_n].$
However, it follows that the degree of any homogeneous element of $Z_c (W, \h)$ is a multiple of $m$
while degree of $f$ is $m.$ Therefore $f$ cannot be an element of  $(V^{+})^2\mathbb{C}[p_1,\cdots,p_n],$
hence a contradiction.
\end{proof}
%
%

%
\begin{proof} [Proof of Theorem \ref{main}] In fact, we will
show the desired result holds for all parabolic subgroups of $W$ with fixed
$\mathfrak{h}=\mathbb{C}^n$. We proceed by induction on $|W|.$ Suppose the statement holds for 
all proper parabolic subgroups of $W$ and we need to prove this for $W.$ 
Let $R$ be the Poisson algebra generated by  $\mathcal{O}\subset\mathbb{C}[\h^*]^{W}$ and $ \mathbb{C}[\h]^{W}$. 
Suppose that there exists 
$$ 
(v, w)\in \text{Supp}_{\mathbb{C}[\h]^{W}\otimes\mathcal{O}} (Z_c(W,\mathfrak{h})/R)
$$
such that $v\neq 0.$  Hence the Poisson algebra generated by $\mathcal{O}$ and $\mathbb{C}[\h]^{W}$ is not
dense in $\widehat{Z_c(W, \h)}_v.$  Therefore by Lemma \ref{reduction} algebras $\mathcal{O}$ and $\mathbb{C}[\h]^{W_v}$
do not generate  ${Z_{c'}(W_v, \h)}$ as a Poisson algebra. However, $W_v$ is a direct product of the wreath products of cyclic groups, which yields
 a contradiction to the induction assumption. Thus we may assume that 
\begin{equation}
\la{supp2}
\text{Supp}_{} (Z_c(W, \mathfrak{h})/R)\subset \lbrace0\rbrace\times\spec(\mathcal{O}).
\end{equation}
Put $I=\sqrt{\mathrm{Ann}_R(Z_c(W,\mathfrak{h})/R)}$. 
It is easy to see that $I$ is a graded  Poisson ideal of $R$. 
Let $J$ be a maximal graded Poisson ideal in $R$ containing $I$, in particular $\sqrt{J}=J.$ 
We claim that $J$ has a finite codimension in $R.$ 
From maximality of $J$ it follows that $S=R/J$ is a graded Poisson algebra having 
no proper nonzero graded Poisson ideals. Now we prove two properties of this algebra: (1) $S$ is finitely
generated over $\mathcal{O}$; (2) $S=S_0$ and ${\rm dim} \,  S_0 < \infty$

By \eqref{supp2} it follows that $\mathbb{C}[\h]^{W}_{+} \subset I.$ On the other hand $R$ is finitely generated over 
$\mathbb{C}[\h]^{W} \otimes \mathcal{O}$. Therefore  $S$ is finitely generated over $\mathbb{C}[\h]^{W} /
(\mathbb{C}[\h]^{W}_{+}) \otimes\mathcal{O} \cong \mathcal{O}$.

First we will show that all except finitely many negative degrees of $S$ vanish. Since $S$ is finitely generated graded 
module over $\mathcal{O}$, there are  homogenous elements $f_1 , \ldots , f_k \in S$ such that $S = \mathcal{O} f_1 + \ldots + \mathcal{O} f_k$. 
The algebra $\mathcal{O}$ is non-negatively graded hence $S$ may contain at most finitely many negative degrees.   
Since $S$ is reduced,  all negative degree terms must vanish. Let $S_{+}$ be a subset of $S$ of positively graded elements. $S_+$
is proper graded Poisson ideal of $S$ contradicting to the fact $S$ has no such ideal. Therefore $S= S_0$. Since $\mathcal{O} $
is non-negatively graded $ \mathcal{O} f_i $ contain finitely many terms in degree $0$ which implies $\dim (S_0) < \infty$.

This proves that $J$ is a finite codimension (in fact codimension one) proper ideal of $R.$ 
Thus $J'=JZ_c(W, \mathfrak{h})$ is a graded proper involutive ideal
of finite codimension in $Z_c(W, \mathfrak{h})$ containing $\mathbb{C}[\h]^{W}_{+}.$
Let $\m_1,\cdots, \m_s$ be all maximal ideals containing $J'.$ Hence each of $\m_i$ must be graded. In particular
$(\mathbb{C}[\h]^{W}\otimes\mathbb{C}[\h^*]^{W})_+\subset \m_i.$ Then by Lemma \ref{GB} there exists $f \in J'$ such that
$f \notin \m_i^2$. On the other hand by Lemma \ref{key}, $J' \subseteq \m_i^2$ hence a contradiction.
\end{proof}
\begin{corollary}
\label{Pois-gen}
The elements 
$$
\sum_{i=1}^nx_i^{km}  \, \quad \mbox{and} \quad  \sum_{i=1}^n y_i^{km}\,,
 \quad \quad k =1, \, 2, \ldots
$$
generate (as an algebra) $\mathbb{C}[\h]^{W}$ and $\mathbb{C}[\h^*]^{W}$ respectively. As a result they generate $Z_c$ as a Poisson algebra.
\end{corollary}
Now we study an automorphism group of $H_c $ and $Z_c $. Let
$ {\rm Aut}_W (H_c)$ be the group of automorphisms of $H_c $ that fixes $W$.  Now we state 
a slightly modified version of the result which first appeared in \cite[Theorem 5.10]{EG}:  
\begin{theorem}
For $\mu \in \c$ and $k \in \z_{\geq 0}$ the assignments 
\begin{eqnarray*}
\Psi_{k, \mu}   \, :&& w \mapsto w \, , \quad x_i \mapsto x_i + \mu \, \sum_{j=1}^n y_j^{km-1} \, , \quad y_i \mapsto y_i  \, , \\
\Phi_{k, \mu}   \, : && w \mapsto w \, , \quad x_i \mapsto x_i  \,  , \quad y_i \mapsto y_i  + \mu \, \sum_{j=1}^n x_j^{km-1} \, , 
\end{eqnarray*}
define automorphisms of $H_c$. Moreover, the map $ G_m \rightarrow  {\rm Aut}_{W}(H_c)$ given
by $\rho (\psi_{k, \mu})  = \Psi_{k, \mu} \, , \, \rho (\phi_{k, \mu})  = \Phi_{k, \mu}$ defines a homomorphism of groups.
\end{theorem}
\begin{proof}
Similar to the proof of Lemma \ref{embGm}.
\end{proof}

Composing the above result
with the canonical homomorphism ${\rm Aut}_{W}(H_c) \rightarrow {\rm Aut}_{\c}\, (Z_c)$ we obtain 
$G_m \rightarrow {\rm Aut}_{\c} (Z_c)$. More explicitly, by  Corollary \ref{Pois-gen}, on generators of $Z_c$,
we have: 
\begin{corollary}
\label{chered-gen}
The action of $G_m$ on $Z_c$ is given by
\begin{eqnarray*}
\sum_{i=1}^nx_i^{km} & \mapsto& \sum_{i=1}^n \Big( x_i + \mu \, \sum_{j=1}^n y_j^{km-1} \Big)^{km},\\
\sum_{i=1}^n y_i^{km} &  \mapsto & \sum_{i=1}^n \Big( y_i + \mu \, \sum_{j=1}^n x_j^{km-1} \Big)^{km},
\end{eqnarray*}
for $k= 1, 2, \ldots \, $ 
\end{corollary}
  This defines an action of $G_m$ on $X_c$.  Next,  we state and prove an equivariant version 
of Etingof-Ginzburg isomorphism \cite[Theorem 11.17(ii), Theorem 11.29(ii)]{EG}, which identifies the generalized 
Calogero-Moser spaces with $X_c$.
\begin{theorem}
\la{pois-iso}
There is a $G_m$-equivariant isomorphism
$$
\Theta : X_c  \, \rightarrow \,  \CC^{\tau}_n(Q_m).
$$
\end{theorem}
\begin{proof}
According to   \cite[Theorem 11.29(ii)]{EG}, there is an isomorphism $\Theta^* : \O ( \CC^{\tau}_n(Q_m)) \rightarrow Z_c$ given by
$$
\mathrm{Tr}(X^k) \, \mapsto \, \sum_{i=1}^nx_i^{km}  \, , \quad 
\mathrm{Tr}(Y^k) \, \mapsto \, \sum_{i=1}^n y_i^{km} .
$$
Let $\sigma$ be either one of the elements of $G_m$: $\psi_{k, \mu} , \, \phi_{k, \mu}$. Then using Lemma \ref{embGm} and 
Corollary \ref{chered-gen}, it is easy to see that the following diagram commutes
$$
 \begin{tikzcd}
   \O ( \CC^{\tau}_n(Q_m))  \arrow[d, "\sigma"] \arrow[r, "\Theta^*"]&
   Z_c \arrow[d, "\sigma"]\\
    \O ( \CC^{\tau}_n(Q_m))  \arrow[r, "\Theta^*"]&Z_c.
    \end{tikzcd}
   $$
Since $G_m$ is generated  by automorphisms $\psi_{k, \mu} , \, \phi_{k, \mu}$, 
the diagram commutes for all $\sigma \in G_m$. This completes the proof.
\end{proof}

Thus the question of transitivity of $G_m$ on $\CC^{\tau}_n(Q_m)$ reduces to that on $X_c $.
The following is our main result.

\begin{theorem}\label{transitivity}
If $X_c$ is smooth, then $G_m$ acts transitively on $X_c.$
\end{theorem}
\begin{proof}
Using the Poisson isomorphism $\Theta$ in Theorem \ref{pois-iso}, we can easily see that 
$\{ \, \Psi_{k, \mu}, \Phi_{k, \mu} \}_{k \geq 1, \, \mu \in \c}$ are Hamiltonian 
flows on $X_c$ for Hamiltonians $\sum_{i=1}^nx_i^{km}, \sum_{i=1}^ny_i^{km}$. By Corollary \ref{Pois-gen}
elements $\{\sum_{i=1}^nx_i^{km}\}_{k \geq 1}$ and  $\{\sum_{i=1}^ny_i^{km}\}_{k \geq 1}$ Poisson generate $Z_c$. Therefore using 
Lemma \ref{flows} we conclude that $G_m$ acts transitively on $X_c$.
\end{proof}
%
%
%
%
%
%
%
%
\section{Appendix. Proof of Theorem~\ref{genfuncmain}}
 For $i, j \in \mathbb{Z}_{\ge 0}$
 and $1 \le k \le m-1$, we set
 $$ A:=Y_{m-1} \ldots Y_1 Y_0\, , \quad B:= X_0 X_1 \ldots X_{m-1} \ , $$
$$C_k:=Y_{m-1} \ldots Y_{m-k}\, X_{m-k} \ldots X_{m-1} .$$

 \begin{proposition}
 \label{propwav1}
 Let $a_n\ldots a_1$ be a closed path at $0$ in $\Pi^{\tau}$. Then
 $$ w(a_n\ldots a_1)v= w(A^i C_k B^j)v +F\big( w(A^{i'}C_{k'}B^{j'})v) $$
 where $m(i+j)+2k=n$ and $F$ is a polynomial with $(i',j',k')$
 such that $n'<n$.
 \end{proposition}
 \begin{proof}
We prove by induction on $n$. The idea of our proof is following.
Note each $a_i$ is either $X_{i_0}$ or $Y_{i_1}$ for some $0 \le i_0, i_1 \le m-1$.
Using relations \eqref{qv1}-\eqref{qv3}, we rearrange $a_t$'s in 
$a_n \ldots a_1$ to get a new closed path $a_n' \ldots a_1'$ at $0$
so that its first part $a_{mj+k}' \ldots a_1'$ is represented by $X_t$'s 
and the second part $a_n' \ldots a_{mj+k+1}'$ by $Y_t$'s.
Then it is easy to see $a_n' \ldots a_1'=A^i C_k B^j$.

Assume that $a_{l} \ldots a_1$ are represented by $Y_t$'s
and $a_{l+1}$ by some $X_t$. Then
$a_{l+1} a_{l} \ldots a_1 =X_{r-1} Y_{r-1} \ldots Y_0 A^q$, where 
 $l=qm+r$ for some $q \ge 0$ and $ 1\le r \le m-1 $.
 Our goal is to move $X_{r-1}$ to the right end of the path. 
To this end we repeatedly use \eqref{qv2} and then \eqref{qv1} 
and obtain
 \begin{eqnarray*}
&& w(a_n\ldots  X_{r-1} Y_{r-1} \ldots Y_0 A^q) v =
w(a_n\ldots Y_{r-2} X_{r-2} Y_{r-2}\ldots Y_0 A^q) v+ \mbox{ l.l.t.} \\[0.1cm]
&=&\ldots = w(a_n\ldots Y_{r-2} Y_{r-3} \ldots Y_0 X_0 Y_0 A^q) v+ \mbox{ l.l.t.} \\[0.1cm]
&=& w(a_n\ldots Y_{r-2} Y_{r-3} \ldots Y_0 Y_{m-1} X_{m-1} A^q)v+\mbox{ l.l.t.},
\end{eqnarray*}
where l.l.t. stands for a polynomial with variables that are functionals $w(-)v$ evaluated 
at closed paths of the length strictly less than $n$, which we can disregard by induction hypothesis.
For $X_{m-1} A^q$, once again using \eqref{qv2} and \eqref{qv1}, we  get
\begin{eqnarray*}
X_{m-1} A^q &=& (X_{m-1} Y_{m-1}) Y_{m-2} \ldots Y_0\,A^{q-1}\\[0.1cm]
&=& Y_{m-2} X_{m-2} Y_{m-2}\ldots Y_0 A^{q-1} + \ldots \\[0.1cm]
&=&Y_{m-2} \ldots Y_{0} X_0 Y_0 A^{q-1} + \ldots \\[0.1cm]
&=&Y_{m-2} \ldots Y_{0}Y_{m-1}X_{m-1}A^{q-1} +Y_{m-2} \ldots Y_{0} w vA^{q-1}+... 
\end{eqnarray*}
and hence, by repeating this procedure $q-1$ times,  we have
%
%
$$ w(a_n\ldots X_{r-1} Y_{r-1} \ldots Y_0 A^q) v
 =w(a_n\ldots Y_{r-2} \ldots Y_{0}A^{q}Y_{m-1} X_{m-1})v+\mbox{l.l.t}.$$
%
We then proceed to $a_{i+2}$. Then either  $a_{i+2}=Y_{r-1}$ or $X_{r-2}$.
If $a_{i+2}=Y_{r-1}$, we consider $a_{i+3}$, and if  $a_{i+2}=X_{r-2}$, we
will move $X_{r-2}$ to the right as we did for $X_{r-1}$.
 \end{proof}
 Similarly, we can show
 \begin{proposition}
\label{propwav2}
For a closed path $a_n\ldots a_1$ at $0$ in $\Pi^{\tau}$, we have
 $$ w(a_n\ldots a_1)v= w(A^i C_k B^j)v +G\big(\Tr(A^{i'}C_{k'}B^{j'}), w(A^{i'}C_{k'}B^{j'})v)\ , $$
 where $m(i+j)+2k=n$ and $F$ is a polynomial with $(i',j',k')$
 such that $n'<n$.
 \end{proposition}
  
 Define, for $ 0 \le i \le m-1$ and $1 \le k \le m-1$,
 $$ C_{i,k}\,:= \begin{cases} 
      Y_{m-i-1} \ldots Y_{m-i-k} X_{m-i-k} \ldots X_{m-i-1},  &  \mbox{ if } i+k < m , \\[0.2cm]
     Y_{m-i-1} \ldots Y_0\, C_{i+k-m} \,X_0 \ldots X_{m-i-1}, &  \mbox{ if  } i+k \ge m . \\
   \end{cases} $$
Note that $C_{0,k}=C_k$. Then we can prove
\begin{lemma} 
\label{lemmatrik}
We have
$$\Tr(C_{i,k})=\Tr(C_k)+g(w\,C_l\,v,\Tr(C_{l}))_{l<k}$$
for some polynomial $g$.
\end{lemma}
 \begin{proof}
 We prove by induction on $k$. For $k=1$, we use \eqref{qv2} and have
 \begin{eqnarray*}
  \Tr(C_{i,1}) &=& \Tr(Y_{m-i-1}X_{m-i-1})=\Tr(X_{m-i}Y_{m-i})-\lambda_{m-i}n_{m-i} \\[0.1cm]
  &=&\ldots = \Tr(C_1)-(\lambda_{m-i}n_{m-i}+ \ldots +\lambda_{m-1}n_{m-1}) \ . 
    \end{eqnarray*}
  Now, for $k\ge 2$ and $i+k< m$, we have
 \begin{eqnarray*}
\Tr(C_{i,k}) &=& \Tr(C_{i+1,k-1} X_{m-i-1} Y_{m-i-1})\\[0.1cm]
&=& \Tr(C_{i+1,k-1}Y_{m-i-2}X_{m-i-2})+\lambda_{m-i-1} \Tr(C_{i+1,k-1})= \ldots \\[0.1cm]
&=&\Tr(C_{i+1,k})+\sum_{s=i+1}^{i+k}\lambda_{m-s}\, \Tr(C_{i+1,k-1})
\end{eqnarray*}
Thus, for $i+k \le m$, we get 
\begin{eqnarray*}
\Tr(C_{i,k})&=&\Tr(C_{i-1,k})-\sum_{s=i}^{i+k-1}\lambda_{m-s}\, \Tr(C_{i,k-1})= \ldots \\[0.1cm]
&=& \Tr(C_k) - \sum_{t=1}^i\,\sum_{s=t}^{t+k-1}\lambda_{m-s}\, \Tr(C_{t,k-1})
\end{eqnarray*}
and then the statement easily follows by the induction assumption.

 Next, for $i+k\ge m$, after using \eqref{qv1}, \eqref{qv2} and Proposition~\ref{propwav1}, we obtain
 \begin{eqnarray*}
\Tr(C_{i,k}) &=& \Tr(Y_{m-i-2}\ldots Y_0\,C_{i+k-m} \,X_0 \ldots X_{m-i-2} X_{m-i-1}Y_{m-i-1})\\[0.1cm]
&=&\Tr(Y_{m-i-2}\ldots Y_0\, C_{i+k-m}\,X_0  \ldots X_{m-i-2} Y_{m-i-2}X_{m-i-2})\\[0.1cm]
&+&\lambda_{m-i-1}\,\Tr(C_{i+1,k-1}) = \ldots \\[0.1cm]
&=&\Tr(Y_{m-i-2}\ldots Y_0\, C_{i+k-m}\, Y_{m-1} X_{m-1}  X_0\ldots X_{m-i-2} )\\[0.1cm]
&+&(\lambda_0+ \lambda_1+\ldots \lambda_{m-i-1})\,\Tr(C_{i+1,k-1})\\[0.1cm]
&-&w( X_0\ldots X_{m-i-2} \, Y_{m-i-2}\ldots Y_0\, C_{i+k-m})v\\[0.1cm]
&=& \Tr(C_{i+1,k})+c_1\cdot \Tr(C_{i+1,k-1})-f(w\,C_{s}\, v)_{s<k}
\end{eqnarray*}
for some $c_1\in \c$ and a polynomial $f_1$.
Then, for $i+k>m$, one has
\begin{eqnarray*}
\Tr(C_{i,k})&=&\Tr(C_{i-1,k})-c_1\cdot \Tr(C_{i,k-1})+f_1(w\,C_{s}\, v)_{s<k}
\end{eqnarray*}
 and, by induction,
 \begin{eqnarray*}
\Tr(C_{i,k})&=&\Tr(C_{i-1,k})+g_1(w\,C_{s}\, v,\Tr(C_{s}))_{s<k}\, .
\end{eqnarray*}
Since this holds for all $i$ such that $i+k>m$ we get
 \begin{eqnarray*}
\Tr(C_{i,k})&=&\Tr(C_{m-k,k})+g_2(w\,C_{s}\, v,\Tr(C_{s}))_{s<k}\, .
\end{eqnarray*}
Now if we use the first part of the proof for $\Tr(C_{m-k,k})$, then we obtain
 \begin{eqnarray*}
\Tr(C_{i,k})&=&\Tr(C_k)+g(w\,C_{s}\, v,\Tr(C_{s}))_{s<k}\, .
\end{eqnarray*}
This finishes our proof.
 \end{proof}
 \begin{proposition} 
 \label{trantrwv}
 For a closed path $a_n\ldots a_1$ at $0$ in $\Pi^{\tau}$, we have
$$(\lambda_0+\ldots +\lambda_{m-1})\,  \Tr(a_n\ldots a_1) = 
w(a_n\ldots a_1)v+ F\big( w(b_{n'} \ldots b_1)v \big),$$
where $b_{n'} \ldots b_1$'s  are closed paths at $0$ with $n'<n$.
 \end{proposition}
 \begin{proof}
 The shall prove by induction on $n$. Without loss of generality 
 we may assume $a_n\ldots a_1$ is either $A^i$ or $B^j$ or $C_k$, 
 since the general case can proved in a similar fashion.
 
 By multiplying \eqref{qv1} by $A^i$ and then taking the trace, we get
$$ wA^i v= \lambda_0 \, \Tr(A^i) + \Tr(A^iY_{m-1}X_{m-1}) - \Tr(A^iX_0Y_0)\, .$$
Then using \eqref{qv2}, we have
\begin{eqnarray*}
\Tr(A^iX_0Y_0) &=& \Tr [ A^{i-1} Y_{m-1}\,...\,Y_1(Y_0 X_0) Y_0] \\[0.1cm]
&=&\Tr [ A^{i-1} Y_{m-1}\,...\,(Y_1X_1)Y_1Y_0] -\lambda_1\, \Tr (A^i)\\[0.1cm]
&=& \ldots = \Tr ( A^{i-1} Y_{m-1}X_{m-1} \, A) -(\lambda_1+\ldots \lambda_{m-1})\, \Tr (A^i)\\[0.1cm]
&=& \Tr(A^i  Y_{m-1}X_{m-1}) - (\lambda_1+\ldots \lambda_{m-1})\, \Tr (A^i)\, .
\end{eqnarray*}
This implies%
$$ w A^i v=(\lambda_0+ \ldots +\lambda_{m-1}) \, \Tr(A^i)\, .$$
 Similarly, we can show 
 $$w B^j v=(\lambda_0+ \ldots +\lambda_{m-1}) \, \Tr(B^j)$$
 and 
 \begin{equation}
 \label{wckv4}
wC_k v= \lambda_0 \, \Tr(C_k) + \Tr(C_kY_{m-1}X_{m-1}) - \Tr(C_kX_0Y_0)\, .
\end{equation}
Now $\lambda_0 \, \Tr(C_k) + \Tr(C_kY_{m-1}X_{m-1})$ is equal to
\begin{eqnarray*}
 && \lambda_0 \, \Tr(C_k) +\Tr(Y_{m-1} \ldots Y_{m-k} X_{m-k} \ldots X_{m-1} Y_{m-1} X_{m-1})\nonumber \\[0.1cm]
 &=&\Tr(C_{k+1})+\sum_{s=0}^k \lambda_{m-s}\, \Tr(C_k)=\Tr(C_{1,k}X_{m-1}Y_{m-1})+
 \sum_{s=0}^k \lambda_{m-s}\, \Tr(C_k)\nonumber \\[0.1cm]
 &=&\Tr(C_{2,k}X_{m-2}Y_{m-2})+\sum_{s=1}^{k+1} \lambda_{m-s}\, 
 \Tr(C_{1,k})+\sum_{s=0}^k \lambda_{m-s}\, \Tr(C_k)\nonumber \\[0.1cm]
 &=&\Tr(C_{m-k,k}X_{k}Y_{k})+\sum_{t=0}^{m-k-1}\,\sum_{s=t}^{t+k} \lambda_{m-s}\,  
 \Tr(C_{t,k}) = \Tr(C_{k-1}X_0Y_0)\nonumber\\[0.1cm]
 &-& \sum_{t=0}^{k-1} w(X_0 \ldots X_{t}
 Y_{t}\ldots Y_0 Y_{m-1}\ldots Y_{m-k+t+1} X_{m-k+t+1}\ldots X_{m-1})v \nonumber\\[0.1cm]
 &+&\sum_{t=0}^{m-k}\,\sum_{s=t}^{t+k} \lambda_{m-s}\,  
 \Tr(C_{t,k}) +\sum_{t=m-k+1}^{m-1} \bigg( \sum_{s=t}^{m} \lambda_{m-s}
 +\sum_{s=1}^{t-m+k} \lambda_{m-s}\bigg) \Tr(C_{t,k}).
 \end{eqnarray*}
Using Lemma~\ref{lemmatrik}, Proposition~\ref{propwav1},
the induction assumption and the following simple identity
 $$\sum_{t=0}^{m-k}\,\sum_{s=t}^{t+k} \lambda_{m-s}
 +\sum_{t=m-k+1}^{m-1} \bigg( \sum_{s=t}^{m} \lambda_{m-s}
 +\sum_{s=1}^{t-m+k} \lambda_{m-s}\bigg)=(k+1)\sum_{s=0}^{m-1} \lambda_s $$ %
 we obtain that, for  $\lambda_0 \, \Tr(C_k) + \Tr(C_kY_{m-1}X_{m-1})$,
 $$ \Tr(C_{k-1}X_0Y_0)-k \cdot wC_kv+
 (k+1)\sum_{s=0}^{m-1} \lambda_s \cdot \Tr(C_k)+F(w C_l v)_{l<k}\, \ .$$
 Finally, by substituting the last expression into \eqref{wckv4} we obtain our statement.
 %
%
%
 \end{proof}
 
\begin{proof}[Proof of Theorem~\ref{genfuncmain}]
We know that the set $\{\Tr(a_n\ldots a_1)\}$, where $a_n\ldots a_1$ run over all
closed paths at $0$, generate $\mathcal{O}(\mathfrak{M}^{\tau})$.
Then, by Proposition~\ref{trantrwv},  $\mathcal{O}(\mathfrak{M}^{\tau})$
is generated by $\{w(a_n\ldots a_1)v\}$. Next, by Proposition~\ref{propwav1} and \ref{propwav2},
$\{w(A^i C_k B^j)v\}$ are generators of the algebra. Finally, using  Proposition~\ref{trantrwv}
for $w(A^i C_k B^j)v$, we get $\{\Tr(A^i C_k B^j)\}$ generate 
$\mathcal{O}(\mathfrak{M}^{\tau})$.
\end{proof}

\bibliographystyle{amsalpha}

\end{document}